\documentclass[11pt,reqno]{amsart}
\usepackage[margin=1.08in,letterpaper]{geometry}   
\usepackage{setspace}
\usepackage{tikz}
\usetikzlibrary{graphs, positioning}
\usepackage{comment}
\usepackage{amsmath, amssymb, amsfonts, enumerate}
\usepackage{amsthm} 
\usepackage{mathtools} 
\usepackage[colorlinks=true, pdfstartview=FitV, linkcolor=blue, citecolor=blue, urlcolor=blue]{hyperref}
\usepackage[capitalise]{cleveref}
\crefformat{equation}{(#2#1#3)}
\crefrangeformat{equation}{(#3#1#4--#5#2#6)}
\crefformat{enumi}{(#2#1#3)}
\crefrangeformat{enumi}{(#3#1#4--#5#2#6)}

\makeatother

\newcounter{intro}

\newtheorem{theorem}[subsection]{Theorem}
\newtheorem{proposition}[subsection]{Proposition}
\newtheorem{lemma}[subsection]{Lemma}
\newtheorem{corollary}[subsection]{Corollary}

\theoremstyle{definition}
\newtheorem{example}[subsection]{Example}
\newtheorem{setting}[subsection]{Setting}
\newtheorem{definition}[subsection]{Definition}
\newtheorem{notation}[subsection]{Notation}
\newtheorem{chunk}[subsection]{}

\newtheorem{conjecture}[subsection]{Conjecture}

\theoremstyle{remark}
\newtheorem{remark}[subsection]{Remark}

\numberwithin{equation}{section}

\newcommand{\qand}{\quad \mbox{ and } \quad }
\newcommand{\qfor}{\quad \mbox{ for } \quad }

\newcommand{\qor}{\quad \mbox{ or } \quad }

\newcommand{\qif}{\quad \mbox{if} \quad}
\newcommand{\qwith}{\quad \mbox{with} \quad}

\newcommand{\qforall}{\quad \mbox{for all} \quad}

\newcommand{\Tor}{\mathrm{Tor}}
\newcommand{\Ext}{\mathrm{Ext}}

\newcommand{\Exp}{\mathrm{Exp}}
\newcommand{\Log}{\mathrm{Log}}
\newcommand{\rate}{\mathrm{rate}}
\newcommand{\reg}{\mathrm{reg}}

\newcommand{\slant}{\mathrm{slant}}
\newcommand{\slope}{\mathrm{slope}}
\newcommand{\kk}{\mathsf k}
\newcommand{\Po}{\mathsf P}
\newcommand{\bo}{\mathsf b}
\newcommand{\Qo}{\mathsf Q}
\newcommand{\pd}{\mathrm{pd}}
\newcommand{\HH}{\mathrm{H}}
\newcommand{\Ker}{\mathrm{Ker}}
\newcommand{\m}{\mathfrak m}
\newcommand{\n}{\mathfrak n}
\newcommand{\st}{\colon\,}
\DeclareMathOperator*{\maximize}{maximize}
\DeclareMathOperator*{\subjectto}{\mbox{subject to}}

\numberwithin{equation}{section}

\title[Maximal shifts for the residue field]{A study of maximal shifts in the minimal graded free resolution of the residue field}

\author[D.~Limbu]{Dinesh Limbu}
\address{Department of Mathematics, University of South Carolina, 
Columbia, SC 29208}
\email{dlimbu@email.sc.edu}

\author[L.~M.~\c{S}ega]{Liana M.~\c{S}ega}
\address{Division of Computing, Analytics and Mathematics, 
University of Missouri-Kansas City, Kansas City, MO 64110 U.S.A.}

\email{segal@umkc.edu}
\author[A.~Vraciu]{Adela Vraciu}
\address{Department of Mathematics, University of South Carolina, 
Columbia, SC 29208}
\email{vraciu@math.sc.edu}

\date{\today}

\begin{document}

\begin{abstract}
We study two invariants, rate and slant, that describe the extremal behavior of the maximal shifts in the minimal free resolution of the residue field over a standard graded algebra. We provide bounds and computations of these invariants for a variety of rings, including compressed level algebras, rings defined by general forms and rings defined by monomials. We also prove an asymptotic additivity property of the maximal shifts for certain classes of rings, including Golod rings, which allows to interpret slant as a limit.   
\end{abstract}


\maketitle
\section{Introduction}

Let $\kk$ be a field, $n\ge 1$ an integer and $R=\kk[x_1,\dots, x_n]/I$ a standard graded singular $\kk$-algebra defined by a homogeneous ideal $I$. Consider the augmented minimal free resolution of $\kk$ over $R$ 
\[
\cdots \to \bigoplus_{j\in \mathbb Z}R(-j)^{\beta_{i,j}}\to \cdots\to \bigoplus_{j\in \mathbb Z}R(-j)^{\beta_{1,j}}\to R\to \kk\to 0\,,
\]
where the notation $R(-j)$ stands for the ring $R$ with shifted grading, namely $R(-j)_a=R_{a-j}$. In this paper we study the {\it maximal shifts} $t_i(R)$, defined by 
\[
t_i(R)=\max\{j\ge 0\colon \beta_{i,j}\ne 0 \} \qfor i\ge 0.
\]
When $R$ is Koszul, $t_i(R)=i$ for all $i\ge 1$ and when $R$ is not Koszul the sequence $\{t_i(R)-i\}_{i\ge 1}$ is unbounded, as proved in \cite{avramov2001peeva}. To further understand the behavior of maximal shifts for non-Koszul rings, the following two invariants are used:
\[
\rate(R)=\sup_{i\ge 2}\frac{t_i(R)-1}{i-1}\qand 
\slant(R)=\sup_{i\ge 1}\frac{t_i(R)}{i}\,.
\]
These invariants were introduced by Backelin \cite{backelin2006rates}, and then studied  by several other authors in   \cite{eisenbud1994initial}, \cite{gasharov2000rationality}, \cite{aramova1995rate}, \cite{conca2001rate}, \cite{avramov2010free}, \cite{ahangari2017rate}, \cite{boij2025rate}, see also the survey \cite{PM}. The existing results give bounds for these invariants and identify classes of rings for which rate is {\it minimal}, i.e. it is one less than the largest degree of a minimal generator of $I$. Minimality of rate is a property that can be viewed as a generalization of the Koszul property to rings that are not necessarily defined by quadrics. 

Similar to the notion of minimal rate, we introduce a notion of {\it minimal slant} (see \cref{mindef}). A ring that has minimal slant must also have minimal rate, but the converse does not hold, most notably for rings defined by monomials. We prove that slant (and hence rate) is minimal for several classes of rings.   Namely, we show $R$ has minimal slant under any of the following assumptions, where $\tau$ denotes the smallest degree of a minimal generator of $I$: 
\begin{enumerate}
\item  The algebra $\Ext_R(\kk,\kk)$ is generated by its elements of degree $1$ and $2$ (\cref{p:12}); 
\item $I$ is an $(r,\tau)$ general symmetric ideal (as defined in \cite{AS}) with $r\ge 1$, $\tau\ge 2$ (\cref{gensym});
\item $R$ is an Artinian compressed level algebra with $\tau\ge 4$  and  $n\ge 5$ 
(\cref{c:compressed}). 
\end{enumerate}

When the defining ideal $I$ of $R$ is generated by $N$ quadrics with generic (i.e.~algebraically independent) coefficients, it is known that $R$ is Koszul, and thus $\rate(R)=\slant(R)=1$, if and only if $N\le n$ or $N\ge\frac{n^2}{4}+\frac{n}{2}$, see \cite[Theorem 7.1]{froberg2002koszul}. We expand this result in two separate directions under assumptions on $R$ (to be made precise later) that hold or are expected to hold when the forms defining $I$ are chosen to be general. 

Firstly, we explore whether similar statements can be obtained when $I$ is generated by $N$ forms of higher degree. Namely, we show in \cref{N&rate} that rate and slant are minimal for large values of $N$, and conversely minimal rate/slant forces bounds on $N$. This corollary is a consequence of \cref{p:genrateslant}, which provides computations for rate and slant that depend the Castelnuovo-Mumford regularity of $R$. 

Secondly, when $R$ is quadratic but not Koszul, we study how the number $N$ of quadrics affects the values taken by rate and slant, by providing some concrete computations (\cref{p:range}) and studying the asymptotic behavior of the invariants when the number of variables increases (\cref{c:limit}).
These results rely on \cref{firstdrop}, which shows that, when $R_{\geqslant 3}=0$, the rate and slant of $R$ are attained at $\alpha=\min\{i\mid t_i(R)>i\}$.

The rate and slant of $R$ describe extremal behavior of maximal shifts. A natural question is whether they also describe the asymptotic behavior of the shifts. This is the topic of \cref{s:asymptotic}. Assuming there exists a surjective Golod homomorphism  $P\to R$, where $P=S$ (thus $R$ is Golod) or $P=S/J$ is a quadratic complete intersection ring, we show in \cref{t:additivity} that there exists an integer $i^*$ such that 
$t_i^R(\kk)=t_{i-i^*}^R(\kk)+t_{i^*}^R(\kk)$ for all $i\gg 0$. It follows from here that, under the given assumptions, slant also describes asymptotic behavior, namely:
\[
 \slant(R)=\lim_{i\rightarrow\infty}\frac{t_i^R(\kk)}{i}=\lim_{i\rightarrow\infty}\frac{t_i^R(\kk)-1}{i-1}.
 \]
 The main ingredient in the proof of this result is a version of the classical unbounded Knapsack Problem, \cref{l:UKP}. 

Finally, in \cref{s:monomial} we study rings defined by monomials. As mentioned earlier, such rings are known to have minimal rate but do not necessarily have minimal slant. For such rings, we establish an upper bound for slant, and we discuss when this bound is attained, see \cref{main1} and \cref{T6.8}.

\section{ Preliminaries} 
\label{s:prelims}
Let $\kk$ be a field, $n\ge 2$ an integer and $S=\kk[x_1,\ldots ,x_n]$ a polynomial ring over $\kk$ with variables in degree $1$. We denote by $\m$ the maximal homogeneous ideal of $S$. Let $I$ be a nonzero homogeneous ideal of $S$ contained in $\m^2$ and set $R=S/I$. We denote by $m(I)$ the largest degree of a minimal generator of $I$.

Let $M$ be a finitely generated $R$-module. For each $i\ge 0$, the homology module $\Tor_i^R(M,\kk)$ is a bigraded vector space, and we denote by $\Tor_i^R(M,\kk)_j$ its component in internal degree $j$. The bigraded {\it betti number} $\beta_{i,j}^R(M)$ can be defined as the dimension of this vector space. 
The bigraded {\it Poincar\'e series} of $M$ is the formal power series in two variables
\[
\Po_M^R(t,u)=\sum_{i\ge 0, j\ge a}\beta_{i,j}^R(M)t^iu^j\,.
\]
where $a$ is the smallest degree of a minimal generator of $M$.

We now introduce the invariants of interest for our study. 

\begin{definition} Let $i\ge a$. The {\it maximal shift at step $i$} of $M$ is defined as 
\[
t_i^R(M)\coloneqq\sup\{j\st \beta_{i,j}^R(M) \neq 0\}\,.
\]
Recall that the (Castelnuovo-Mumford) {\it regularity} of $M$ is defined by 
\[
\reg_R(M)=\sup\{t_i^R(M)-i\colon i\ge 0\}\,.
\]
The following invariants are studied in order to understand the asymptotic behavior of maximal shifts when regularity is infinite. 

The {\it rate} of $M$, respectively the {\it slope} of $M$, are defined by 
\[
\rate_R(M)=\sup_{i\ge 2}\frac{t_i^R(M)-1}{i-1}\qand 
\slope_R(M)=\sup_{i\ge 1}\frac{t_i^R(M)-t_0^R(M)}{i}\,.
\]
When $M$ is a ring quotient of $R$, then the slope of $M$ over $R$ is also called {\it slant}, namely
\[
\slant_R(M)=\sup_{i\ge 1}\frac{t_i^R(M)}{i}.
\]

We are particularly interested in the case when $M=\kk$, when we also write 
\[
t_i(R)=t_i^R(\kk), \quad \rate(R)=\rate_R(\kk) \qand \slant(R)=\slant_R(\kk)=\slope_R(\kk)\,.
\]
\end{definition}

 A general upper bound for slope can be found in \cite{aramova1995rate}. The following inequalities are known from \cite{avramov2010free} and \cite{backelin1988relations} respectively
\begin{align}
\label{sss}
 \slope_R(M)&\le  \max\{\slope_S(M), \slope_S(R)\}\\
 \label{sr}
 \slant_S(R)&\le \rate(R)+1.
\end{align}
In particular, \eqref{sss} with $M=\kk$ gives 
$\slant(R)\le \slant_S(R)$.

Additionally, the following inequalities follow directly from the definitions: 
\begin{equation}
\label{e:slant-rate}
   \slant(R)\le \rate(R)\le 2\,\slant(R)-1. 
\end{equation}

The ring $R$ is said to be {\it Koszul} if $\kk$ has a linear resolution over $R$, or in other words, $t_i^R(\kk)=i$ for all $i\ge 0$. Clearly,  $R$ is Koszul if and only if $\rate(R)=1$ if and only if $\slant(R)=1$.

\begin{chunk}
\label{Tate}
It is well-known, see \cite{avramov1998infinite}, that a free resolution of $\kk$ over $R$ can be constructed by a process of repeated adjunction of DG-algebra variables; this resolution is called the {\it Tate resolution} of $\kk$. While the Tate resolution is usually described over a local ring, the construction carries over in the graded setting, resulting in a minimal graded free resolution. We denote it $R\langle X\rangle$, where $X=\cup_{i\ge 1}X_i$ is the set of adjoined variables, with $X_i$ denoting the set of variables of homological degree $i$. Since $R\langle X\rangle$ is a minimal free resolution of $\kk$, $t_i^R(\kk)$ is equal to the largest internal degree of a product of $X$-variables which has homological degree $i$. 

By construction, the variables $T_1, \dots, T_n$ of the set $X_1$ satisfy $\partial(T_i)=x_i$ for all $i\in [n]$, and hence they have internal degree $1$. The images of the variables of $X_2$ under the differential form a basis of the first homology of the complex $R\langle X_1\rangle$, which is the Koszul complex of $R$ on $x_1, \dots x_n$. Since $H_1(R\langle X_1\rangle)\cong I\otimes_S\kk$ (see \cite[Lemma 4.1.3]{avramov1998infinite}), the largest internal degree of a variable of $X_2$ is equal to $m(I)$, the largest degree of a minimal generator of $I$. Further, a basis of $R\langle X\rangle\otimes_R\kk$ in homological degree $2$ consists of the images of the variables of $X_2$, together with the images of the products $T_iT_j$ with $i\ne j$, $i,j\in [n]$, which have internal degree $2$.  Consequently, the largest internal degree of an element of $\Tor_2^R(\kk,\kk)$ is $m(I)$. Thus, one has
\begin{equation}
\label{1&2}
t_1^R(\kk)=1\qand t_2^R(\kk)=m(I).
\end{equation}
\end{chunk}
In particular, the second equality of \eqref{1&2} yields inequalities
\begin{equation}
\label{less}
\rate(R)\ge m(I)-1 \qand 
\slant(R)\ge \frac{m(I)}{2}.
\end{equation}

\begin{definition}
\label{mindef}
We say that $R$ has {\it minimal rate} if $\rate(R)=m(I)-1$. We say that $R$ has {\it minimal slant} if $\slant(R)=\frac{m(I)}{2}$. 
\end{definition}

In view of \eqref{1&2}, we see that $\rate(R)$ is minimal precisely when the supremum in the definition of rate is attained at $i=2$  and $\slant(R)$ is minimal precisely when the supremum in the definition of slant is attained at $i=2$.  It is clear from \eqref{e:slant-rate} that if $R$ has minimal slant, then it also has minimal rate, but the converse is not true. Indeed, rings defined by monomials are known to have minimal rate, but, as we will see in \cref{s:monomial}, they do not necessarily have minimal slant.

 Several classes of rings with minimal rate are known: monomial ideals \cite{eisenbud1994initial},  generic toric rings \cite{gasharov2000rationality}, the coordinate ring corresponding to a set of points in projective space \cite{conca2001rate}, and compressed Artinian Gorenstein algebras \cite{boij2025rate}. Our first result below records a large class of rings that have minimal slant (and hence minimal rate).

It is known that $\Ext_R(\kk,\kk)$ has a structure of (bi)graded $\kk$-algebra with the Yoneda product. The algebra $R$ is Koszul if and only if its Yoneda algebra $\Ext_R(\kk,\kk)$ is generated by its elements of cohomological degree $1$, see \cite[Theorem 1.2]{Lofwall}. When $R$ is a complete intersection, meaning that the ideal $I$ is generated by a regular sequence, $\Ext_R(\kk,\kk)$ is generated by its elements of degree $1$ and $2$,  see \cite{sjodin1976}. There exist other classes of rings with this property, such as quotients of $S$ by powers of $\m$ \cite{Levin1981} and compressed Gorenstein algebras with even socle degree \cite{HS}. 

\begin{proposition}
\label{p:12}
Let $R$ be such that the Yoneda algebra $\Ext_R(\kk,\kk)$ is generated by its elements of (cohomological) degree $1$ and $2$. Then  $R$ has minimal rate and minimal slant.  
\end{proposition}

\begin{proof}
If $R$ is Koszul, then the statement is clear.  Assume now $R$ is not Koszul, and hence the Yoneda algebra cannot be generated only by its elements of degree $1$. 

Consider a minimal set of homogeneous algebra generators for $\Ext_R(\kk,\kk)$ of (cohomological) degrees $1$ and $2$. The elements in cohomological degree $1$ have internal degree $-1$. Let $2\le d_1\le \dots \le d_N$ such that the minimal generators of cohomological degree $2$ have internal degree $-d_i$ for $i\in [N]$. Since $t_2^R(\kk)=m(I)$, we must have $d_N=m(I)$.  By our hypothesis, $\Ext^i_R(\kk,\kk)$ is generated as a $\kk$-vector space by products of elements of this generating set. Consequently, $t_i^R(\kk)$ is less than or equal to the maximum in the problem
\begin{align}
&\maximize_{q_j, p\in \mathbb Z} \quad p+\sum_{j=1}^Nq_jd_j\qquad \subjectto  \qquad p+ \sum_{j=1}^N2q_j=i, \quad p\ge 0, q_j\ge 0\,.
\end{align}
A solution to this problem is given by 
\[
q_N^*=\Big\lfloor\frac{i}{2}\Big\rfloor, \qquad q_j^*=0 \qfor j\ne N, \qquad p^*=i-2q_N^*\,.
\]
Indeed, for any other choice of integers $p$ and $q_j$ with $j\in [N]$ satisfying the given constraint, we have $\sum_{j=1}^Nq_j\le \lfloor\frac{i}{2}\rfloor$ and hence
\[
p+\sum_{j=1}^Nq_jd_j=i+\Big(\sum_{j=1}^Nq_j\Big)(d_j-2)\le i+\Big\lfloor\frac{i}{2}\Big\rfloor (d_N-2)= p^*+\sum_{j=1}^Nq_j^*d_j\,.
\]
We have thus
\[
t_i^R(\kk)\le (i-2q_N^*)+q_N^*d_N=i+\lfloor\textstyle{\frac{i}{2}}\rfloor (d_N-2)
\]
and hence
\begin{align*}
\slant(R)&\le \sup_{i\ge 1}\Big \{\frac{t_i^R(\kk)}{i}\Big\}=1+(d_N-2)\sup_{i\ge 1}\Big\{\frac{\lfloor\textstyle{\frac{i}{2}}\rfloor}{i}\Big\}=1+(d_N-2)(\textstyle{\frac{1}{2}})=\frac{m(I)}{2}.
\end{align*}
Using \eqref{less}, we see that equalities must hold, so $R$ has minimal slant.  Then \eqref{e:slant-rate} implies $R$ has minimal rate as well.
\end{proof}

\begin{remark} We observe that a converse to the conclusion of  \cref{p:12} does not hold. Namely, if $R$ has minimal rate and slant, it does not follow that $\Ext_R(\kk,\kk)$ is generated by its elements of degree $1$ and $2$. Indeed, if $R$ is a generic Gorenstein Artinian $\kk$-algebra in the sense of \cite{boij2025rate}, with socle in degree $\rho\ge 5$ and $\rho$ odd, then $I$ is generated in a single degree by \cite[Corollary 3.6]{boij2025rate} and then \cite[Corollary 3.7]{HS} shows that $\Ext_R(\kk,\kk)$ is not generated by its elements of degree $1$ and $2$. On the other hand, it is proved in \cite{boij2025rate} that $R$ has minimal rate and we will see in \cref{c:compressed} that it also has minimal slant. 
\end{remark}

We now give bounds on rate and slant using Koszul homology. Let $(K, \partial^K)$ denote the Koszul complex on a minimal generating set of $I$. The first module of syzygies of $I$ in a minimal free resolution over $S$ can be identified with $\text{Ker}(\partial_1^K)$, and hence $t_2^S(R)$ is the largest degree of a minimal generator of $\Ker(\partial_1^K)$. A minimal generator of this syzygy module is said to be a {\it Koszul syzygy} if it is in the image of $\partial_2^K$ and is said to be a {\it non-Koszul syzygy} otherwise; thus, the classes of the non-Koszul syzygies form a basis of $\HH_1(K)$. The next result shows that the largest degree of a non-Koszul syzygy of $I$ can be used to bound rate. 

\begin{proposition}
\label{non-koszul}
 With notation as above, we have 
\[
\rate(R)\geq \frac{t_0^S(\HH_1(K))-1}{2} \qand \slant(R)\geq \frac{t_0^S(\HH_1(K))}{3}\,.
\]
\end{proposition}


\begin{proof}
Let $R\langle X\rangle$ be a minimal Tate resolution of $\kk$ over $R$  (see \ref{Tate}) and $S[Y]$  a minimal model of $R$ over $S$; we refer to \cite{avramov1998infinite} for the terminology.  By \cite[7.2.6]{avramov1998infinite}, the set of variables $X_3$ of homological degree $3$ in the acyclic closure is in bijection with the set of variables $Y_2$ of homological degree $2$ in the minimal model, which correspond to a basis of $H_1(K)$. The proof of this result carries through in our graded setting, in which case the bijection preserves internal degrees. 

Set  $e=t_0^S(\HH_1(K))$. As seen above, $e$ is the largest degree of a variable in $X_3$. It follows then from the discussion in \cref{Tate} that $t_3^R(\kk)\ge e$, yielding the desired inequalities. 
\end{proof}

 Finally, we recall the concepts of Golod ring and Golod homomorphism, which will play an important role in our results.  We refer to \cite{avramov1998infinite}
 for details. 

\begin{chunk}
\label{Golod-def}
Let $P=S/J$, where $J$ is a homogeneous ideal of $S$. If $\varphi\colon P\to R$ is a surjective homomorphism of graded rings, then the following coefficient-wise inequality holds:  
\[
\Po_\kk^R(t,u)\le \frac{\Po^P_\kk(t,u)}{1-t(\Po_R^P(t,u)-1)}.
\]
The map $\varphi$ is said to be a {\it Golod homomorphism} if equality holds above. If $\varphi$ is Golod and $J=0$, we say that $R$ is a {\it Golod ring}.

If $P$ is a complete intersection ring (i.e. its defining ideal is generated by a regular sequence) and $\varphi\colon P\to R'$ is a surjective Golod homomorphism of graded rings, then $\Po_\kk^{R'}(t,u)=\frac{(1+tu)^n}{\bo _{R'}(t,u)}$, where $\bo_{R'}(t,u)\in \mathbb Z[t,u]$ and the Poincar\'e series of every finitely generated $R'$-module can be written as a rational function, with denominator $\bo_{R'}(t,u)$, see \cite[Proposition 5.18]{AKM}. 

Furthermore, if $\varphi\colon P\to R'$ and $\varphi'\colon R'\to R$ are surjective Golod homomorphisms of graded rings and $P$ is a complete intersection, then, after writing the Poincar\'e series over $R'$ as fractions with denominator $b_{R'}(t,u)$ and simplifying $b_{R'}(t,u)$, we have
\begin{equation}
\label{nice-fraction}
\Po_\kk^R(t,u)=\frac{\Po_{\kk}^{R'}(t,u)}{1-t(\Po_R^{R'}(t,u)-1)}=\frac{(1+tu)^n}{\bo _R(t,u)} \qwith \bo _R(t,u)\in \mathbb Z[t,u].
\end{equation}
\end{chunk}

\section{Upper bounds for rate and slant}

In this section we give a strategy for computing rate and slant of $R$ when the Poincar\'e series of $\kk$ over $R$ is a rational function of a specific form (\cref{denom-poly}) and also when $R_{\geqslant 3}=0$ (\cref{firstdrop}). We then use the Serre inequality to provide in \cref{t:upperb} concrete upper bounds for rate and slant. To streamline our exposition, we have found it useful to state several intermediate results in a more general setting, as described next.

The maximal shifts of $\kk$ over $R$ and the rate and slant of $R$ can be read from the Poincar\'e series $\Po_\kk^R(t,u)$, and as such we find it useful for our arguments to define these invariants more generally, for any formal power series $\Po(t,u)\in \mathbb Z[u][[t]]$, where
\begin{align}
\label{P-def}
\begin{split}
\Po(t,u)=1+\sum_{i\ge 1}B_i(u)t^i& \qwith \deg B_i\begin{cases}\in \{-\infty, 1\} &\qif i=1\\\ge i &\qif B_i\ne 0\,.
\end{cases}
\end{split}
\end{align}

Let $\mathcal P$ denote the set of all power series as in \eqref{P-def}.

\begin{lemma}\label{inverse}
Let $\displaystyle \Po=1+\sum_{i\ge 1} B_i(u)t^i$. Then 
$\displaystyle \Po^{-1}=1+\sum_{i\ge 1} C_i(u)t^i$, where
\begin{equation*}
C_i(u)=\sum_{\substack{i_1+\dots+i_k=i\\i_1, \dots, i_k\geq 1,k\geq 1}}(-1)^k B_{i_1}(u) \cdots B_{i_k}(u)
\end{equation*}
for all $i\ge 1$. In particular, it follows that $\mathcal P$ is a group with multiplication. 
\end{lemma}
\begin{proof}
Let $\displaystyle \Po^{-1}=\sum_{i\ge 0} C_i(u)t^i$. One can easily see that $C_0(u)=1$. 
For all $i\ge 1$, we have 
\[
\sum_{p+q=i,\,\, p, q\ge 0} B_p(u)C_q(u)=0\,,
\]
where we set $B_0(u)=1$. Therefore 
$$ C_i(u)=-B_i(u)-\sum_{p+q=i,\,\, p,q\ge 1} B_p(u)C_q(u).$$
    The claim now follows by induction on $i$.
\end{proof}

For $P\in \mathcal P$, we set:
\begin{align*}
\quad t_i(\Po)=\deg B_i\qquad  &\rate(\Po)=\sup_{i\ge 2}\frac{t_i^R(\Po)-1}{i-1}\qquad 
\slant(\Po)=\sup_{i\ge 1}\frac{t_i(\Po)}{i}\,.
\end{align*}
 
 Note that \eqref{P-def} implies that if $\Po\in \mathcal P$, then $t_i(\Po)=-\infty$ or else $t_i(\Po)\ge i$ (with equality when $i=1$), and hence $\rate(\Po)$ and $\slant(\Po)$ are either equal to $-\infty$ or else they are greater than or equal to $1$. Also, note that if $\Po(t,u)=\Po_\kk^R(t,u)$, then $\Po\in \mathcal P$ and $\rate(\Po)=\rate(R)$ and $\slant(\Po)=\slant(R)$.

\begin{definition}
\label{defs}
Let $\Po\in \mathcal P$. 
\begin{itemize}
\item Let $i\ge 2$. We say that $\rate(\Po)$ is {\it attained at $i$} if  $\displaystyle 
\rate(\Po)=\frac{t_{i}(\Po)-1}{i-1}\ne -\infty\,.
$
We say that $\rate(\Po)$ is attained if it is attained at some $i\ge 2$. 
\item Let $i^*\ge 2$. We say that $\rate(\Po)$ is {\it attained for the first time at $i^*$} if $i^*$ is the smallest of all values $i\ge 2$ such that $\rate(\Po)$ is attained at $i$. 
\item Let $i\ge 1$. We say that $\slant(\Po)$ is {\it attained at $i$} if
$\displaystyle 
\slant(\Po)=\frac{t_{i}^R(\Po)}{i}\ne -\infty\,.
$
We say that $\slant(\Po)$ is attained if it is attained at some $i\ge 1$.
\item Let $i^*\ge 1$. We say that $\slant(\Po)$ is {\it attained for the first time at $i^*$} if $i^*$ is the smallest of all values $i\ge 1$ such that $\slant(\Po)$ is attained at $i$.  
\end{itemize}
When working with  $\rate(R)$ and $\slant(R)$, we adapt the terminology by replacing $\Po$ with $R$ in the definitions. For example, we say $\rate(R)$ is attained at $i$ if $\rate(R)=(t_i^R(\kk)-1)/(i-1)\ne -\infty$. 
\end{definition}

Recall from \eqref{e:slant-rate} that $\slant(R)\le\rate(R)$ and both slant and rate are equal to $1$ when $R$ is Koszul. We show next that $\slant(R)<\rate(R)$ when $R$ is not Koszul and $\slant(R)$ is attained. 
\begin{proposition}
If $\slant(R)$ is attained, then $\slant(R)=\rate(R)$ if and only if $R$ is Koszul.
\end{proposition}

\begin{proof}
Assume  $\rate(R)=\slant(R)$ and let $i$ be such that $\slant(R)$ is attained at $i$. If $R$ is not Koszul, then $\slant(R)\ne 1$ and hence $i>1$. Further, we have
\begin{equation}
\label{e:ineqsrr}
\frac{t_i^R(\kk)-1}{i-1}\le \rate(R)=\slant(R)=\frac{t_i^R(\kk)}{i}\le \frac{t_i^R(\kk)-1}{i-1}\,,
\end{equation}
where the first inequality comes from the definition of rate and the second inequality comes from the fact that $t_i^R(\kk)\ge i$. All inequalities above must thus be equalities, and in particular
\[
\frac{t_i^R(\kk)}{i}= \frac{t_i^R(\kk)-1}{i-1}\,,
\]
implying $t_i^R(\kk)=i$. Plugging back into \eqref{e:ineqsrr}, it follows that $\slant(R)=1$, and hence $R$ is Koszul. The reverse implication is clear. 
\end{proof}

We examine the behavior of the rate and slant when taking products and inverses, paying attention to cases when the rate and slope are attained. 

\begin{lemma}
\label{l:inverse}
   If $\Po\in \mathcal P$, then 
   \begin{enumerate}
    \item $\slant(\Po^{-1})=\slant(\Po)$. Furthermore, $\slant(\Po)$ is attained for the first time at $i^*$ if and only if  $\slant(\Po^{-1})$ is attained for the first time at $i^*$. 
    \item $\rate(\Po)\le 1$ if and only if $\rate(\Po^{-1})\le 1$. 
       \item If $\rate(\Po)> 1$, then $\rate(\Po^{-1})=\rate(\Po)$. Furthermore, $\rate(\Po)$ is attained for the first time at $i^*$ if and only if $\rate(\Po^{-1})$ is attained for the first time at $i^*$. 
   \end{enumerate}
   \end{lemma}
   \begin{proof}
       We write $\Po=\sum_{p\ge 0}B_p(u)t^p$ and $\Po^{-1}=1+\sum_{i\ge 1}C_i(u)t^i$. We then have
  \begin{align}
  \label{C}
C_i(u)&=\sum_{\substack{i_1+\dots+i_k=i\\i_1, \dots, i_k\geq 1,k\geq 1}}(-1)^k B_{i_1}(u) \cdots B_{i_k}(u)\qforall i>0.
  \end{align} 
  Consequently, for $i\ge 1$ we have 
  \begin{equation}
  \label{Pinv}
  t_i(\Po^{-1})=\deg(C_i)\le \max\{\deg(B_{i_1}\dots B_{i_k})\colon i_1+\dots+i_k=i, i_1, \dots, i_k\geq 1,k\geq 1\}.
  \end{equation}
  The inequality in \eqref{Pinv} may be strict due to cancellations among the terms of the summation in \eqref{C}. However, equality holds if the maximum in \eqref{Pinv} is attained for exactly one choice of the integers $k, i_1, \dots, i_k$; this observation will be used in our arguments below.

  Set $\gamma=\slant(\Po)$ and $\alpha=\rate(\Po)$. We have thus $\deg(B_i)\le \gamma i$ and $\deg(B_i)\le \alpha(i-1)+1$ for all $i\ge 1$. (The last inequality holds when $i=1$ because $\Po\in \mathcal P$.)
  
(1)  Let $i\ge 1$. For all $k\ge 1$ and positive integers $i_1, \dots, i_k$ with $i_1+\dots+i_k=i$:
  \begin{equation}
  \label{B}
    \deg (B_{i_1}\dots B_{i_k})\le \sum_{p=1}^k\gamma i_p=\gamma i.
  \end{equation}
 This shows $\deg(C_i)\le\gamma i$ for all $i\ge 1$, and hence $\slant(\Po^{-1})\le \gamma=\slant(\Po)$. Replacing $\Po$ with $\Po^{-1}$, we see that the reverse inequality holds, and hence 
 $\slant(\Po^{-1})=\slant(\Po)$. Observe that equality holds in \eqref{B} if and only if $\slant(\Po)$ is attained at $i_1, \dots, i_k$. Consequently, if $\slant(\Po)$ is attained for the first time at $i^*$, then $\deg(C_i)<\gamma i$ when $i<i^*$. Additionally,  $\deg(C_{i^*})=\deg(B_{i^*})=\gamma i^*$ because for $i=i^*$ the maximum in \eqref{Pinv} is attained for only one choice of indices, namely $k=1$ and $i_1=i^*$. It follows that $\rate(\Po^{-1})$ is attained for the first time at $i^*$. Replacing $\Po$ with $\Po^{-1}$, we obtain the converse of the ``if and only if'' statement. 

(2) Since $\Po\in \mathcal P$, note that $\rate(\Po)\le 1$ if and only if $\slant(\Po)\le 1$. Thus, the statement is a consequence of (1). 

(3) Assume $\alpha> 1$. By (2), we also know that $\rate(\Po^{-1})>1$. Let $i\ge 1$. For all $k\ge 1$ and positive integers $i_1, \dots, i_k$ with $i_1+\dots+i_k=i$:
\begin{align}
\begin{split}
\label{degree-prod}
\deg (B_{i_1}\dots B_{i_k})&\le \sum_{p=1}^k\left(\alpha(i_p-1)+1\right)= \alpha(i-k)+k=\alpha i+k(1-\alpha)\\
&\le \alpha i+1-\alpha=\alpha (i-1)+1.
\end{split}
\end{align}
Thus $\deg C_i\le \alpha (i-1)+1$ for all $i\ge 1$. This implies $\rate(\Po^{-1})\le \alpha=\rate(\Po)$. Replacing $\Po$ with $\Po^{-1}$, we also get the reverse inequality, and hence $\rate(\Po^{-1})=\rate(\Po)$. The inequalities in \eqref{degree-prod} are equalities precisely when $k=1$, hence $i_1=i$, and $\rate(\Po)$ is attained at $i$. Thus, if $\rate(\Po)$ is attained at $i^*$ and  $i<i^*$,  we have $\deg(C_i)<\alpha (i-1)+1$. Additionally, $\deg(C_{i^*})=\deg(B_{i^*})=\alpha(i^*-1)+1$ because when $i=i^*$ the maximum in \eqref{Pinv} is attained for only one choice of indices, namely $k=1$ and $i_1=i^*$. This shows $\rate(\Po^{-1})$ is attained for the first time at $i^*$. Replacing $\Po$ with $\Po^{-1}$, we obtain the converse of the ``if and only if'' statement. 
   \end{proof}

\begin{lemma}
\label{QP}
 Let $\Po, \Qo\in \mathcal P$. 
  \begin{enumerate}
    \item If $\slant(\Qo)\le\slant(\Po)$, then $\slant(\Po\Qo)\le\slant(\Po)$. \\
    Moreover, if $\slant(\Qo)<\slant(\Po)$ and  $\slant(\Po)$ is attained for the first time at $i^*$,  then $\slant(\Po\Qo)=\slant(\Po)$ and $\slant(\Po\Qo)$ is attained for the first time at $i^*$.  
    \item If $\max\{1,\rate(\Qo)\}\le \rate(\Po)$, then $\rate(\Po\Qo)\le \rate(\Po)$. \\Moreover, if $\max\{1,\rate(\Qo)\}<\rate(\Po)$ and  $\rate(\Po)$ is attained for the first time at $i^*$,  then $\rate(\Po\Qo)=\rate(\Po)$ and $\rate(\Po\Qo)$ is attained for the first time at $i^*$.  
    \end{enumerate}
\end{lemma}

\begin{proof}
  We write $\Po=\sum_{p\ge 0}B_p(u)t^p$, $\Qo=\sum_{q\ge 0}C_q(u)t^q$ and  $\Po\Qo=\sum_{i\ge 0}D_i(u)t^i$. We then have
  \begin{align}
  \label{e:DBC}
D_i(u)&=\sum_{{p+q=i, \,\, p,q\ge 0}}B_p(u)C_q(u) \qforall i\ge 0.
  \end{align}
Observe that \eqref{e:DBC} gives an inequality
\begin{equation}
\label{e:max-deg}
t_i(\Po\Qo)=\deg(D_i)\le \max\{\deg(B_pC_q)\colon p+q=i, p\ge 0,q\ge 0\}.
\end{equation}
The inequality above may be strict due to cancellations between the terms of the summation in \eqref{e:DBC}. However, \eqref{e:max-deg} is an equality if the maximum is attained for exactly one choice of $(p,q)$. This observation will be used in our arguments below. 

(1) Set $\eta=\slant(\Qo)$ and $\gamma=\slant(\Po)$ and assume $\eta\le \gamma$.  Then $\deg(B_p)\le \gamma p$ for all $p\ge 0$ and $\deg(C_q)\le \eta q$ for all $q\ge 0$. 
Let $i\ge 0$. For $p,q\ge 0$ so that $p+q=i$ we have
 \begin{equation}
  \label{e:BCpq}
  \deg(B_pC_q)\le \gamma p+\eta q\le \gamma p+\gamma q=\gamma i,
  \end{equation}
implying $\deg(D_i)\le \gamma i$ for all $i\ge 0$, and hence $\rate(\Po\Qo)\le \gamma=\rate(\Po)$. Assume now $\eta<\gamma$. Then, the second inequality in \eqref{e:BCpq} is an equality if and only if $q=0$. In this case, the first inequality in \eqref{e:BCpq} is an equality if and only if $\rate(\Po)$ is attained at $p=i$. Thus, if $i<i^*$, we see that $t_i(\Po)<\gamma i$. On the other hand, $t_{i^*}(\Po)=\gamma i^*$, since the maximum is attained in \eqref{e:max-deg} only for the tuple $(p,q)=(i,0)$. This shows that $\slant(\Po\Qo)=\gamma=\slant(\Po)$ and $\slant(\Po\Qo)$ is attained for the first time at $i^*$. 

(2) Set $\beta=\rate(\Qo)$ and $\alpha=\rate(\Po)$ and assume $\max\{1,\beta\}\le \alpha$.  Then $\deg(B_p)\le \alpha(p-1)+1$ for all $p\ge 1$ and $\deg(C_q)\le \beta (q-1)+1$ for all $q\ge 1$. 
Let $i\ge 2$. For $p,q\ge 1$ so that $p+q=i$ we have
 \begin{align}
 \begin{split}
  \label{e:BCpq2}
  \deg(B_pC_q)\le \alpha(p-1)+\beta(q-1)+2&\le \alpha(p-1)+\alpha(q-1)+2= \alpha(i-2)+2\\
 & = \alpha(i-1)+(2-\alpha)\le \alpha(i-1)+1.
  \end{split}
  \end{align}
  On the other hand if $p+q=i$ and $p=0$ or $q=0$ we have
\begin{align}
  \label{e:BCpq3}
  \deg(B_0C_q)&\le \beta(i-1)+1\le\alpha(i-1)+1\\
  \label{e:BCpq4}
  \deg(B_pC_0)&\le\alpha(i-1)+1.
  \end{align}
It follows that $\deg(D_i)\le \alpha(i-1)+1$ for all $i\ge 2$, and hence $\rate(\Po\Qo)\le \alpha=\rate(\Po)$. 

 Assume now $\max\{1,\beta\}<\alpha$.  Then equality never holds in \eqref{e:BCpq2} or in \eqref{e:BCpq3} and holds in \eqref{e:BCpq4} exactly when $\rate(\Po)$ is attained at $p=i$. Thus, if $\rate(\Po)$ is attained for the first time at $i^*$, it follows that $t_i(\Po\Qo)<\alpha(i-1)+1$ when $i<i^*$ and  $t_{i^*}(\Po\Qo)=\alpha(i^*-1)+1$. This shows that $\rate(\Po\Qo)$ is attained at $i^*$. 
   \end{proof}

\cref{QP} and \cref{l:inverse} allow for some concrete computations of rate and slant, and we record two such applications next. 

Let $\n$ denote the maximal homogeneous ideal of $R$. For each $i$ we let $\mathcal S_{\n^i}$ denote the cokernel of the Yoneda product
\[
\Ext_R(\kk,\kk)\otimes\Ext_R^1(R/\n^i,\kk)\to \Ext_R(R/\n^i, \kk)\,.
\]
The ring $R$ is said to have the property $\mathcal M_3$ if  $\mathcal S_{\n^i}=0$ for all $i\ge 2$. This property can be regarded as generalizing the property $R_{\geqslant 3}=0$.
Among rings that satisfy $\mathcal M_3$ are also rings $R=S/I$ with $I$ defined by quadrics that satisfy $\dim_\kk R_3\le 2$, see \cite[Theorem B.9]{roos1994}. 

\begin{proposition}\label{firstdrop}
Assume $R$ satisfies $\mathcal M_3$ (e.g.~$R_{\geqslant 3}=0$) and is not Koszul. Set 
$$
\alpha:=\mathrm{\min}\{i\, | \, t_i^R(\kk)>i \}. 
$$
Then both $\rate(R)$ and $\slant(R)$ are attained for the first time at $\alpha$ and 
$$
\mathrm{rate}(R)=1+\frac{1}{\alpha -1}\qand \slant(R)=1+\frac{1}{\alpha}.
$$
\end{proposition}
\begin{proof}
From \cite[Theorem 2.3]{MR846457}, we have 
\begin{equation}\label{eq}
\Po_\kk^R(t, u)=\frac{tA(tu)}{1+t-A(tu)H(-tu)}
\end{equation}
where $H(z)$ denotes the Hilbert series of $R$, and $A(z)$ is the Hilbert series of the Koszul dual, i.e. the subalgebra of $\Ext_R(\kk,\kk)$ generated by its elements of degree $1$.
We write
\[
A(z)H(-z)=1+z^a\sum_{i\ge 0}c_iz^i
\]
with $c_i\in \mathbb Z$ such that $c_0\ne 0$, and $a\ge 1$. 
We define $\Po$ as indicated below. 
\begin{equation}
\label{defineP}
\Po:=\frac{1+t-A(tu)H(-tu)}{t}=1-\sum_{i\ge 0}c_it^{i+a-1}u^{i+a}.
\end{equation}
Set $\Qo(t,u)=A(tu)$. Then $\slant(\Qo)=1$ and $\Po_\kk^R=\Qo\Po^{-1}$. We must have then $\Po\in \mathcal P$ and in particular $t_1(\Po)\le 1$, hence $a\ge 3$. Using the second equality in \eqref{defineP}, we see 
\begin{equation}
\label{e:a}
\rate(\Po)=\frac{a-1}{a-2} \qand \slant(\Po)=\frac{a}{a-1}
\end{equation}
and both are first attained at $a-1$. 
By \cref{QP} and \cref{l:inverse} we have $\rate(\Po_\kk^R)=\rate(\Qo\Po^{-1})=\rate(R)$ and $\slant(\Po_\kk^R)=\rate(\Qo\Po^{-1})=\slant(R)$. It remains to notice that $\alpha=a-1$. Indeed, this can be seen by expanding $\Qo\Po^{-1}$.  
\end{proof}

\begin{proposition}
\label{denom-poly}
    If $R$ is such that  $\Po_\kk^R(t,u)=\frac{(1+tu)^n}{\bo _R(t,u)}$ and $\bo _R(t,u)\in \mathbb Z[t,u]$,  then 
    \[
    \rate(R)=\rate(\bo _R) \qand \slant(R)=\slant(\bo _R).
    \]
    Furthemore, $\rate(R)$ and $\slant(R)$ are attained.  More precisely, if $\rate(\bo_R)$ is attained for the first time at $i^*$, then $\rate(R)$ is attained for the first time at $i^*$, and if $R$ is not Koszul and $\slant(\bo_R)$ is attained for the first time at $i^*$, then $\slant(R)$ is attained for the first time at $i^*$. 
\end{proposition}

\begin{proof}
Observe that $\Po_\kk^R(t,u), (1+tu)^n\in \mathcal P$ and hence $\bo _R(t,u)\in \mathcal P$ since $\mathcal P$ is a group with multiplication. Since $R$ is assumed to be singular, we have $\bo_R\ne 1$. 

Since $\Po_\kk^R(t,u)=1+ntu \text{ mod }(t^2)$, we have $\bo_R(t,u)=1  \text{ mod }(t^2)$, and hence $t_1(\bo_R)=-\infty$. In particular, $\rate(\bo_R)\ge 1$ and $\slant(\bo_R)\ge 1$. If $\rate(\bo_R)=1$ (equivalently, $\slant(\bo_R)=1$), then an expansion of the fraction shows that $R$ is Koszul, and the statement is clear. 

Since $\bo_R(t,u)$ is a polynomial, it attains both its rate and its slant. Assume $\rate(\bo_R)>1$ and $\slant(\bo_R)>1$. We apply \cref{l:inverse}(3) with $\Po=\bo _R(t,u)$ to conclude $\slant(\bo_R^{-1})=\slant(\bo_R)>1$ and $\rate(\bo_R^{-1})=\rate(\bo_R)>1$. Then we apply  \cref{QP}(2) with $\Qo=(1+tu)^n$ and $\Po=\bo _R(t,u)^{-1}$, noting that $\slant(\Qo)=\rate(\Qo)=1$. 
\end{proof}

\begin{remark}
\label{golodgolod}
The hypothesis of \cref{denom-poly} holds when there exist surjective graded Golod homomorphisms $P\to R'$ and $R'\to R$ with $P$ a graded complete intersection ring, or $R$ is defined by a monomial ideal. Indeed, see \cref{Golod-def} for the first two cases and \cite{berglund2006poincare} for the monomial case. 
\end{remark}

We now use \cref{denom-poly} to give bounds on slant and rate, using Serre's inequality. 

\begin{theorem}
\label{t:upperb}
For each $k$ with $0\le k<\pd_SR$ the following inequalities hold: 
\begin{align}
\label{t1}
\rate(R)&\le \max\Big\{\frac{t_i^S(R)-1}{i} \st 1\le i\le \pd_SR \Big\}\\
\label{t2}
&\le \max\Big\{1+\frac{\reg_S(R)-1}{k+1}, \frac{t_i^S(R)-1}{i} \st 1\le i\le k \Big\}\\
\label{t3}
\slant(R)&\le \max\Big\{\frac{t_{i}^S(R)}{i+1} \st 1\le i\le \pd_SR \Big\}\\
\label{t4}
&\le \max\Big\{1+\frac{\reg_S(R)-1}{k+2}, \frac{t_{i}^S(R)}{i+1} \st 1\le i\le k \Big\}.
\end{align}
Equalities hold in \eqref{t1} and \eqref{t3}  when $R$ is a Golod ring.  
\end{theorem}

\begin{proof}
    We use the Serre inequality 
    \begin{equation}
    \label{fraction}
    \Po_\kk^R(t,u)\le \frac{\Po_\kk^S(t)}{1-t(\Po_R^S(t,u)-1)}=\frac{(1+tu)^n}{1-t(\Po_R^S(t,u)-1)}\,.
    \end{equation} 
Set $\bo _R(t,u)$ to be the denominator of the fraction above, noting that $\bo_R\in \mathcal P$. Observe that $t_1(\bo _R)=-\infty$ and $t_{i+1}(\bo _R)=t_i^S(R)$ for all $i\ge 1$. Thus, the expressions on the right in \cref{t1}  and \cref{t3} are exactly the slope, respectively, slant, of $\bo _R$, which according to the computation in the proof \cref{denom-poly} are also the slope, respectively, slant, of the fraction in \eqref{fraction}. The inequalities in \cref{t1} and \cref{t3} follow thus from the Serre inequality.

Set $d=\pd_SR$ and let $0\le k<d$.  Let $A$ denote the right-hand side of \eqref{t1} and $A_k$ denote the right-hand side of \eqref{t2}. To show the inequality in \eqref{t2} we need to show that $A\le A_k$. To see this, we argue below that $A\le A_{d-1}$ and $A_u\le A_{u-1}$ for all $u$ with $1\le u\le d-1$. 

Observe
\begin{equation}
\label{3ineq}
\frac{t_{i}^S(R)-1}{i}\le \frac{\reg_R(S)+i-1}{i}=1+\frac{\reg_R(S)-1}{(i-1)+1}
\end{equation}
for all $i$ with $1\le i\le d$. 
With $i=d$, \eqref{3ineq} shows $A\le A_{d-1}$. Now let $1\le u\le d-1$ and observe 
\begin{equation}
\label{ureg}
\frac{\reg_R(S)-1}{u+1}<\frac{\reg_R(S)-1}{u}.
\end{equation}
Combining \eqref{ureg} and \eqref{3ineq} with $i=u$, we see that $A_u\le A_{u-1}$.

Now let $B$ denote the right-hand side of \cref{t3} and $B_k$ denote the right-hand side of \cref{t4}. 
To show the inequality in \eqref{t4} we need to show that $B\le B_k$. To see this, we argue below that $B\le B_{d-1}$ and $B_u\le B_{u-1}$ for all $u$ with $1\le u\le d-1$. 
Observe 
\begin{equation}
\label{moreineq}
\frac{t_{i+1}^S(R)}{i+2}\le \frac{\reg_R(S)+i}{i+1}=1+\frac{\reg_R(S)-1}{(i-1)+2}
\end{equation}
for all $i$ with $1\le i\le d$. 
With $i=d$, \eqref{moreineq} justifies $B\le B_{d-1}$. Now let $1\le u\le d-1$ and observe
\begin{equation}
\label{moreineq2}
\frac{\reg_R(S)-1}{u+2}\le\frac{\reg_R(S)-1}{u+1} 
\end{equation}
Combining \eqref{moreineq2} and \eqref{moreineq} with $i=u$, we see that $B_u\le B_{u-1}$. 
\end{proof}

\cref{t:upperb} can be used to identify rings of minimal rate/slant, as seen below. 

\begin{corollary}
\label{c:minimal}
The ring $R$ has minimal rate under any of the following assumptions:
\begin{enumerate}
\item The maximal shifts of $R$ satisfy
\[
t_i^S(R)-1\le (m(I)-1)i \qforall 2\le i\le \pd_RS.
\]
\item For some $0\le k<\pd_SR$, we have 
\begin{align*}
\reg_S(R)&\le (k+1)m(I)-2k-1\qand t_i^S(R)-1\le (m(I)-1)i
\qforall 2\le i\le k.
\end{align*}
\end{enumerate}

The ring $R$ has minimal slant under any of the following assumptions 

\begin{enumerate}
\item[(3)] The maximal shifts of $R$ satisfy
\[
t_i^S(R)\le \frac{i+1}{2}m(I) \qforall 2\le i\le \pd_RS.
\]
\item[(4)] For some $0\le k<\pd_SR$, we have 
\begin{align*}
\reg_S(R)&\le \frac{k+2}{2}m(I)-k-1\qand
t_i^S(R)\le \frac{i+1}{2}m(I) \qforall 2\le i\le k.
\end{align*} 
\end{enumerate}
\end{corollary}

\section{Rings with minimal slant}

As mentioned in \cref{s:prelims}, several classes of rings are known to have minimal rate, and this property can be understood as a generalization of the Koszul property. Since minimal slant implies minimal rate, but not conversely, minimal slant is a better refinement of homological behavior that is close to Koszulness. In this section, we identify classes of rings that exhibit this behavior. A first class was already established in \cref{p:12}. We continue with a study of rings defined by general symmetric ideals and compressed level algebras, using the results of the previous section together with existing literature on these classes of rings. 

In \cite[Theorem 7.10]{AS}, with the terminology defined therein, the authors show that if the ideal $I$ defining $R$ is an {\it $(r, \tau)$-general symmetric ideal} then $R$ is Koszul when $\tau=2$ and when $\tau>2$ it satisfies the following property: 
\begin{align}
\begin{split}
\label{symmetric-betti}
 \quad \beta_{i,j}^S(R)=0  \qif &j\notin \{i+\tau, i+\tau-1\}\qand i\ge 1,  \qor\\
&j=i+\tau \qand 1\le i\le n-2. 
\end{split}
\end{align}
While we refer to \cite{AS} for the definition of such ideals, we note here that $r$ denotes the number of minimal generators up to symmetry and $\tau$ is the degree of these generators. 

We show that such rings have minimal rate and slant. 

\begin{corollary}
\label{gensym}
Let $r\ge 1$ and $\tau\ge 2$ be two integers. If $R$ is defined by an $(r,\tau)$-general symmetric ideal $I$ of $S$, then $R$ has minimal rate and slant. 
\end{corollary}
\begin{proof}
We have $m(I)=\tau$. As noted above, if $\tau=2$ then $R$ is Koszul. Assume now $\tau\ge 3$. Condition \eqref{symmetric-betti} gives
\[
t_i^S(R)\le \begin{cases} i+\tau &\qif i\in \{n-1,n\}
\\
 i+\tau-1 &\qif   1\le i\le n-2
\end{cases}
\]
and hence
\[
\frac{t_i^S(R)}{i+1}\le \frac{\tau+i}{i+1}\le \frac{\tau}{2}=\frac{m(I)}{2}  \qif \tau>3 \qor i>2.
\] 
In order to apply \cref{c:minimal}(3) to conclude that $\slant(R)$ is minimal we want to further argue that an inequality between the first and the last term above holds also when $\tau=3$ and $i=2$, even if the second inequality above does not hold. 

Assume thus $\tau=3$ and $i=2$. If $n\ge 4$, then $t_2(R)= 2+3-1=4$ and the inequality becomes $4/3\le 3/2$, a true statement. Assume thus $n=3$.  In this case, we want to argue that we also have $t_2(R)\le 4$, or equivalently $\beta_{2,2+\tau}^S(R)=0$. According to \cite[Theorem 7.2]{AS}, when $\tau=3$ we have 
\[
\beta_{2,5}(R)=\max\{P(3)-P(2)-r,0\}
\]
where $P(d)$ denotes the number of partitions of $d$ with at most $n=3$ parts. Observing that $P(3)=3$ and $P(2)=2$, we conclude that $\beta_{2,5}(R)=0$ in this case. Thus, we can apply  \cref{c:minimal}(3) to conclude that $R$ has minimal slant. Finally, as noted in \cref{s:prelims}, minimal slant implies minimal rate. 
\end{proof}

We now discuss compressed level algebras, whose definition we recall next. If $R$ is Artinian, then $R$ is said to be {\it level} if its socle is generated in a single degree. 
Assume $R$ is an Artinian level algebra quotient of $S$ with $s$-dimensional socle in degree $\rho$. 
Then $R$ is {\it compressed} if its Hilbert function is given by 
\begin{equation}
\label{compressed-def}
\dim_\kk (R_d) = \min\{\dim_{\kk} (S_d), s\dim_{\kk} (S_{\rho-d})\}\,.
\end{equation}
The {\it initial degree} of a homogeneous ideal $I\subseteq S_{\ge 2}$ is the smallest degree of a minimal generator of $I$.  We will also refer to it as the initial degree of the ring $R=S/I$. When $R$ is level compressed, its initial degree can be read from \cref{compressed-def}: it is the smallest $d$ for which $\dim_\kk(S_d)>s\dim_\kk(S_{\rho-d})$.

\begin{corollary}
\label{c:compressed}
    Assume $R=S/I$ an Artinian compressed level with initial degree $\tau$. 
\begin{enumerate}
\item If $n=3$ and $\tau\ge 4$, or $n\ge 4$ and $\tau\ge 3$, then $R$ has minimal rate. 
\item If $n=4$ and $\tau\ge 6$, or $n\ge 5$ and $\tau\ge 4$, then $R$ has minimal slant. 
\end{enumerate}
\end{corollary}

\begin{proof}
By \cite{boij}, we know that for $1\le i\le n-1$, 
\[
\beta_{i,j}^S(R)\ne 0 \implies j \in\{ \tau+i-1, \tau+i\}. 
\]
In particular, 
$
\tau-1\le m(I)-1\le \tau$ and   $t_i^S(R)\le \tau+i$ for all $1\le i\le n-1
$. When $3\le \tau$ and $2\le i<n$,
we have 
\[
\frac{t_i^S(R)-1}{i}\le \frac{\tau+i-1}{i}\le \tau-1\le m(I)-1.
\]
 Taking $d=\tau$ in \cref{compressed-def}, we see that we must have $\dim(S_{\rho-\tau})<\dim(S_{\tau})$, implying $\rho\le 2\tau-1$. Since $t_n^S(R)=n+\rho$, it follows that $t_n^S(R)\le n+2\tau-1$, and, under the assumptions of (1), 
\[
\frac{t_n^S(R)-1}{n}\le \frac{2\tau-2+n}{n}\le \tau-1\le m(I)-1\,.
\]
Then the conclusion of (1) follows from \cref{c:minimal}(1). 

To prove (2), note that for $\tau\ge 4$ we have 
\[
\frac{t_i^S(R)}{i+1}\le \frac{\tau+i}{i+1}\le \frac{\tau}{2} \le \frac{m(I)}{2}.
\]

Also, under the assumptions of (2)
\[
\frac{t_n^S(R)}{n+1}\le \frac{2\tau+n-1}{n+1}\le \frac{\tau}{2} \le \frac{m(I)}{2}.
\]
The conclusion of (2) follows from \cref{c:minimal}(3).
\end{proof}

One can further attempt to understand the remaining cases that are not covered in \cref{c:compressed}, namely when $\tau$ is small. 
For Gorenstein compressed rings we have more information. 
It is known from \cite{boij2025rate} that the rate of a Gorenstein compressed quotient of $S$ is minimal when its socle degree $\rho$ satisfies $\rho\ne 3$. Furthermore, when $\rho=3$ (implying $\tau=2$),  \cite[Theorem 6.3]{CRV} shows that the same conclusion holds (namely $R$ is Koszul) if one assumes that the ring is defined via the Macaulay inverse system by a generic form. On the other hand, there exist graded Gorenstein quadratic rings $R$ with $\rho=3$ (implying $R$ is compressed) which are not Koszul, and thus do not have minimal rate or slant, see for example \cite{MS}.  Further information on slant for Gorenstein compressed rings, that covers cases not previously addressed can be obtained in the Gorenstein case using the formula for $\Po_\kk^R(t,u)$ established in \cite{boij2025rate} and \cref{denom-poly}.

\section{Rate and slant for rings defined by generic forms}

There are several ways in which a collection of homogeneous polynomials in a polynomial ring over a field can be understood as being ``generic''; we will make such definitions clear below. Imposing generic assumptions is a classical theme in commutative algebra, as they often yield uniform, quantifiable and even extremal behavior. In this section we make the case that these phenomena also hold or can be expected to hold in our context. 

We say that $R=S/I$ is a quotient of $S$ of type $(N; d_1,\ldots ,d_N)$ when $I$ can be minimally generated by forms $f_1, \dots, f_N$ with $\deg(f_i)=d_i$ for all $i$, and $2\le d_1\le d_2\le \dots \le d_N$. 
When $d_1=d_2=\cdots =d_N=d$, we use the simplified notation $(N;d)$ for the type. The forms $f_1, \dots, f_N$ are said to have {\it generic coefficients} if the set of all coefficients of these polynomials is algebraically independent over the prime field of $\kk$.

If $f_i$ are as above, consider the product of projective spaces $\mathbb P=\prod_{1\le i\le N}\mathbb P(R_{d_i})$ parametrizing the coefficients of $f_1, \dots, f_N$; when $d_1=\dots=d_N=d$ one may also consider $\mathbb P$ to be $\text{Grass}(R_d,N)$. We say that a {\it general quotient of $S$ of type} $(N;d_1,\dots, d_N)$, satisfies a property $(*)$ when there exists a nonempty open subset $U$ of the parameter space $\mathbb P$ such that, for each choice of $f_1, \dots, f_N$  corresponding to a point in $U$, the ring $S/(f_1, \dots, f_N)$ satisfies $(*)$. 

 A famous conjecture of Fr\"oberg \cite{Froberg-85} predicts that the Hilbert series of a general quotient of $S$ of type $(N;d_1,\dots, d_N)$ is 
 \begin{equation}
 \label{Froberg}
 \Big[\frac{\prod_{i=1}^N(1-z^{d_i})}{(1-z)^n}\Big]
 \end{equation}
 where $[\Po]$ denotes the truncation at the first non-positive term of a power series $\Po$. 
 Using the terminology of \cite{pardue2009syzygies}, a sequence of elements $f_1, \ldots, f_N\in S$ of degrees $d_1, \ldots, d_N$ is called a {\it semi-regular sequence} if the Hilbert series of $S/(f_1, \ldots, f_N)$ is given by (\ref{Froberg}). In this terminology, Fr\"oberg's conjecture states that if $R=S/I$ is a general quotient of type $(N; d_1, \ldots, d_N)$ then $I$ is generated by a semi-regular sequence.

\begin{chunk}
    \label{interval}
The following results are known for quotients defined by quadrics, see (\cite[Theorem 7.1]{froberg2002koszul}):
\begin{enumerate}
\item A general quotient of $S$ of type $(N;2)$ is Koszul when $N\le n$ and is not Koszul when $n<N<\frac{n^2}{4}+\frac{n}{2}$. 
\item A quotient $S/(f_1, \dots, f_N)$ of type $(N;2)$ in  which the forms $f_1, \dots, f_N$ have generic coefficients is Koszul if and only if $N\le n$ or $N\ge \frac{n^2}{4}+\frac{n}{2}$. 
\end{enumerate}
\end{chunk}

Recall that for quotients defined by quadrics, the Koszul property is equivalent to minimal rate. Motivated by the results mentioned above, our goal is to further study rate and slant of general quotients of $S$, with two goals in mind:
\begin{itemize}
\item Understand when minimal rate/slant occurs in non-quadratic cases;
\item Describe the behavior of rate/slant in the non-Koszul quadratic cases when the inequalities  $n<N<\frac{n^2}{4}+\frac{n}{2}$ hold.   
\end{itemize}

If $N<n$, then a general quotient of $S$ of type $(N; d_1, \dots, d_N)$ is a complete intersection, and as such it has minimal slant by \cref{p:12}. A first step towards understanding rate for non-Koszul quadratic general quotients is made in \cite{diethorn}, where the case $N=n+1$ is considered.

In what follows, given a type $(N;d_1, \dots, d_N)$, we consider the following condition: 
\[
(\#) \quad \text{The first nonpositive coefficient in $\frac{\prod_{i=1}^N(1-t^{d_i})}{(1-t)^n}$ is zero.}
\]
\begin{theorem}
\label{p:genrateslant}
Assume $R=S/I$ is of type $(N;d_1, \dots, d_N)$ with $d_i\ge 2$ for all $i$. Let $\rho=\reg_S(R)$. Then the following hold: 
\begin{enumerate}
\item 
If $\rho\le  2m(I)-3$,  then $R$ has minimal rate.
\item If $\rho>2m(I)-3$, $N>n$ and $I$ is generated by a semi-regular sequence, then
\[
\frac{\rho}{2}\le \rate(R)\le \frac{\rho+1}{2}\,.
\]
Furthermore,  
\begin{enumerate}
     \item If $R$ has minimal rate, then $\rho=2m(I)-2$. 
    \item If $(\#)$ holds, then $\rate(R)=\frac{\rho+1}{2}$. 
    \item If $\beta_2^S(R)_{\rho+2}=0$ and $\rho\ge 4$, then $\rate(R)=\frac{\rho}{2}$.
\end{enumerate}
\item If $2\rho\le 3m(I)-4$, then $R$ has minimal slant. 
\item If $2\rho> 3m(I)-4$, $N>n$ and $I$ is generated by a semi-regular sequence, then
\[
\frac{\rho+1}{3}\le \slant(R)\le \frac{\rho+2}{3}\,.
\]
Furthermore, 
\begin{enumerate}
\item If $R$ has minimal slant, then $2\rho=3m(I)-2$ or $2\rho=3m(I)-3$. 
    \item If $(\#)$ holds, then $\slant(R)=\frac{\rho+2}{3}$. 
    \item If $\beta_2^S(R)_{\rho+2}=0$ and $\rho\ge 5$, then $\slant(R)=\frac{\rho+1}{3}$.
\end{enumerate}
\end{enumerate}
\end{theorem}
\begin{proof}
Observe that the condition that $N>n$ tells us that the ideal $I$ defining $R$ is not a complete intersection, and thus it has non-trivial Koszul syzygies. 

(1) The conclusion follows directly from \cref{c:minimal}(2) with $k=1$.

(2) Assume $\rho\ge 2m(I)-2$, $N>n$ and $I$ is generated by a semi-regular sequence. 
By \cite[Theorem 3.6]{pardue2009syzygies}, the minimal free resolution of $R$ over $S$ agrees in the first $\rho-2$ strands with the Koszul complex $K$ on a minimal generating set of $I$. Consequently, the non-Koszul syzygies of $I$ are in degrees at least $\rho+1$. We have thus
\begin{equation}
\label{e:genrate}
\frac{\rho}{2}\le \rate(R)\le \max\Big\{\frac{\rho+1}{2}, m(I)-1\Big\}\le \frac{\rho+1}{2}\,,
\end{equation}
where the first inequality comes from \cref{non-koszul}, the second inequality comes from \eqref{t2} with $k=1$ and the last inequality comes from the hypothesis on $\rho$.

Statement (2)(a) follows directly from the first inequality in \eqref{e:genrate} and the hypothesis on $\rho$. 

 Furthermore, if $(\#)$ holds, the minimal free resolution of $R$ and the Koszul complex $K$ agree in the first $\rho-2$ strands (by \cite[Theorem 3.6]{pardue2009syzygies}), hence the non-Koszul syzygies of $I$ are in degrees at least $\rho+2$. In this case, \cref{non-koszul} gives an inequality
 $\displaystyle 
 \frac{\rho+1}{2}\le \rate(R)
 $,
 which is forced to be an equality by \eqref{e:genrate}. 

Assume now $\beta_2^S(R)_{\rho+2}=0$ and $\rho\ge 4$. Then $t_2^S(R)=\rho+1$ and thus 
$\displaystyle 
\frac{t_2^S(R)-1}{2}=\frac{\rho}{2}\,.
$
Further,  \cref{t2} with $k=2$ and the assumption on $\rho$ yield
\[
\rate(R)\le \max\Big\{\frac{\rho+2}{3}, \frac{\rho}{2}, m(I)-1\Big\}\le \frac{\rho}{2}\,.
\]
In view of the first inequality of \eqref{e:genrate}, the inequalities above are equalities, yielding (2)(c).  

(3) The conclusion follows directly from \cref{c:minimal}(4) with $k=1$.

(4) Assume $2\rho\ge 3m(I)-3$, equivalently $\displaystyle \frac{\rho+1.5}{3}\ge \frac{m(I)}{2}$, $N>n$ and $I$ is generated by a semi-regular sequence. We have then
\begin{equation}
\label{e:genslant}
\frac{\rho+1}{3}\le \slant(R)\le \max\Big\{\frac{m(I)}{2}, \frac{\rho+2}{3}\Big\}=\frac{\rho+2}{3}\,,
\end{equation}
where the first inequality comes from \cref{non-koszul}, using again the fact noted in the proof of (2) that the non-Koszul syzygies of $I$  are in degree at least $\rho+1$, the second inequality comes from \eqref{t4} with $k=1$ and the last inequality comes from the hypothesis on $\rho$.

Statement (4)(a) follows directly from the first inequality in \eqref{e:genslant} and the hypothesis on $\rho$. 

If $(\#)$ holds, then, as noted in the proof of (2), the non-Koszul syzygies of $I$ are in degree at least $\rho+2$, hence \cref{non-koszul} gives an inequality
$\displaystyle 
\frac{\rho+2}{3}\le \slant(R)
$,
which is forced to be an equality by \eqref{e:genslant}. 

Assume now $\beta_2^S(R)_{\rho+2}=0$  and $\rho\ge 5$ and thus
$\displaystyle 
\frac{t_2^S(R)}{3}=\frac{\rho+1}{3}\,.
$
Further, \eqref{t4} with $k=2$ and the assumptions on $\rho$ yield
\[
\slant(R)\le \max\Big\{\frac{\rho+3}{4}, \frac{\rho+1}{3}, \frac{m(I)}{2}\Big\}\le \frac{\rho+1}{3}\,.
\]
In view of the first equality in \eqref{e:genslant}, the inequalities above must be equalities. 
\end{proof}

\begin{remark}
One case when the Fr\"oberg conjecture is known to hold is when $N=n+1$ and the residue field has characteristic $0$ or greater than $n$; this follows from work of Watanabe. In this case, one can see that $(\#)$ holds when $n$ is even. On the other hand, it is shown in \cite{diethorn} that $\beta_2^S(R)_{\rho+2}=0$ when $n$ is odd.  Then \cref{p:genrateslant} provides formulas for rate that recover the computation in \cite{diethorn}, while additionally providing formulas for slant.
\end{remark}

 Assuming Froberg's conjecture, the regularity of $R$ is uniquely determined by the integers $n, N,d_1, \dots, d_N$. In particular, we describe regularity when $d_1=\dots=d_N=d$ as follows: 
 \begin{lemma}
 \label{regineq}
    Assume that $R=S/I$ is of type $(N;d)$ with $N\ge n$ and $d\ge 2$, and $I$ is generated by a semi-regular sequence. Then
    \[
    \reg_S(R)=\min\left\{p\in \mathbb N \colon N \binom{n+p-d}{p+1-d}\ge \binom{n+p}{p+1}\right\}.
    \]    
 \end{lemma}
 \begin{proof}
Since $N\ge n$, we see that $R$ must be Artinian, and hence  $\mathrm{reg}_S(R)$ is equal to the top socle degree of $R$ (i.e. the largest $p$ with $R_p\ne 0$.) 
Let $p\in \mathbb N$. Given the expression in \eqref{Froberg} for the Hilbert function of $R$, the inequality $\reg_R(S)\le p$ is equivalent to the fact that the coefficient of $z^{p+1} $ in $\displaystyle \frac{(1-z^d)^N}{(1-z)^n}$ is non-positive. 
This coefficient equals
\begin{gather*}
\left(\mathrm{coefficient\ of \ } z^{p+1} \ \mathrm{in} \ \frac{1}{(1-z)^n}\right)- N \cdot\left( \mathrm{coefficient\ of \ } z^{p+1-d} \ \mathrm{in} \ \frac{1}{(1-z)^n}\right)\\
=\binom{n+p}{p+1} - N\binom{n+p-d}{p+1-d}. \qedhere
\end{gather*}
\end{proof}

\begin{remark}
\label{r:noFC}
Assume $d=2$ and $N\ge \frac{\textstyle{\binom{n+2}{3}}}{n}$. When $I$ is generated by a semi-regular sequence, \cref{regineq} gives $\reg_S(R)\le 2$, or equivalently $R_{\geqslant 3}=0$. The same conclusion holds when assuming that $R=S/I$ is a general quotient of $S$ of type $(N;d)$ instead of the fact that $I$ is generated by a semi-regular sequence.  
Indeed,  with $I=(f_1, \dots, f_N)$, the fact that $R$ is general implies that the elements $x_if_j$ with $1\le i\le n$ and $1\le j\le \frac{\textstyle{\binom{n+2}{3}}}{n}$ are linearly independent by \cite{hochster1987linear}, yielding that $R_{\geqslant 3}=0$.
\end{remark}

 Combining \cref{p:genrateslant} with \cref{regineq}, we obtain:
\begin{corollary}
\label{N&rate}
If $R=S/I$ is of type $(N;d)$ with $d\ge 2$ and $I$ is generated by a semi-regular sequence, then
\begin{enumerate}
\item If $\displaystyle N\ge \frac{\binom{n+2d-3}{2d-2}}{\binom{n+d-3}{d-2}}$, then $R$ has minimal rate. 

\item If $R$ has minimal rate and $N>n$, then
$N\ge \displaystyle \frac{\binom{n+2d-2}{2d-1}}{\binom{n+d-2}{d-1}}.$

\item If $\displaystyle N\le \frac{\binom{n+\lfloor \frac{3d-4}{2}\rfloor}{\lfloor \frac{3d-2}{2}\rfloor}}{\binom{n+\lfloor \frac{d-4}{2}\rfloor }{\lfloor \frac{d-2}{2}\rfloor}}$, then $R$ has minimal slant.

\item If $R$ has minimal slant and $N>n$, then $\displaystyle N\ge \frac{\binom{n+\lfloor \frac{3d-2}{2}\rfloor}{\lfloor \frac{3d}{2}\rfloor}}{\binom{n+\lfloor \frac{d-2}{2}\rfloor }{\lfloor \frac{d}{2}\rfloor}}$.
\end{enumerate}
\end{corollary} 

 Neither of the implications in the statements of \cref{N&rate} is reversible, as can be seen in the case $d=2$, based on the known range for $N$ when $R$ is Koszul (see \ref{interval}).

For the remainder of the section, we will consider the case when $d=2$, and the inequalities of (2) in \cref{N&rate} hold, but $R$ is not Koszul. In this case, we have $R_{\geqslant 3}=0$, and hence $\rate(R)\le \frac{3}{2}$ and $\slant(R)\le \frac{4}{3}$ by \cref{p:genrateslant}. We show next that there is a range of values of $N$ for which equalities hold. Note that this statement does not assume the Fr\"oberg conjecture.

\begin{proposition}
\label{p:range}
Let $R$ be a general quotient of $S$ of type $(N;2)$ with 
\[
\frac{\textstyle{\binom{n+2}{3}}}{n}\le N<\frac{\sqrt{2n^4-2n^2+1}-n^2+n+1}{2}.
\]
Then $\rate(R)=\frac{3}{2}$ and $\slant(R)=\frac{4}{3}$. 
\end{proposition}
\begin{proof}
The inequality $\frac{\textstyle{\binom{n+2}{3}}}{n}\le N$ implies that $R_{\geqslant 3}=0$ by \cref{r:noFC} and hence the Hilbert series of $R$ is given by 
\begin{equation}
\label{e:HilbR}
H_R(u)=1+nu+(\textstyle{\binom{n+1}{2}}-N)u^2.
\end{equation}
Let $R\langle X\rangle$ be a Tate resolution of $\kk$ over $R$, and denote by $X_{i,j}$ the subset of variables $X$ consisting of variables of homological degree $i$ and internal degree $j$. The (graded) deviations of $R$, denoted $\varepsilon_{i,j}$, are defined as the cardinality of the set $X_{i,j}$. Note that $\varepsilon_{i,j}=0$ when $j<i$. One can then write the Poincar\'e series of $\kk$ as 
\begin{equation}
\label{e:Pdeviations}
\Po_{\kk}^R(t,u)=\frac{\prod_{i\geq 0}\prod_{j\geq 2i+1 }(1+t^{2i+1}u^j)^{\epsilon_{2i+1,j}}}{\prod_{i\geq 1}\prod_{j\geq 2i }(1-t^{2i}u^j)^{\epsilon_{2i,j}}}.
\end{equation}
The deviations in homological degree $1$ and $2$ can be seen to be
\[ 
\varepsilon_{1,j}=\begin{cases} n &\qif j=1\\
0 &\qif j\ne 1
\end{cases}
\qand 
\varepsilon_{2,j}=\begin{cases}
    N& \qif j=2\\
    0 &\qif j\ne 2\,.
\end{cases}
\]
We show next that $t_3(R)>3$, which gives $\alpha=3$ in \cref{firstdrop}, yielding the desired equalities. To this extent, it suffices to show that $\varepsilon_{3,4}>0$.

The Poincar\'e series of $\kk$ and the Hilbert series of $R$ are connected through a formula that is a consequence of the additivity of (vector space) dimension in the minimal free resolution of $\kk$: 
\[
H_R(u)\Po_\kk^R(-1,u)=1\,.
\]
Using \eqref{e:Pdeviations} and \eqref{e:HilbR}, we have thus
\[
\Big(\prod_{i\geq 1}\prod_{j\geq 2i+1 }(1-u^j)^{\epsilon_{2i+1,j}}\Big)(1+nu+(\textstyle{\binom{n+1}{2}}-N)u^2)=\prod_{i\geq 0}\prod_{j\geq 2i }(1-u^j)^{\epsilon_{2i,j}}.
\]
Setting $a=\epsilon_{3,3}$, $b=\epsilon_{3,4}$, and $c=\epsilon_{4,4}$, the following equality then holds modulo $(u^5)$: 
\[
(1-u)^n(1-u^3)^{a}(1-u^4)^b(1+nu+(\textstyle{\binom{n+1}{2}}-N)u^2)=(1-u^2)^{N}(1-u^4)^{c}.
\]
From here, we further obtain:
\begin{gather*}
\left(1-a u^3-nu+nau^4+\textstyle{\binom{n}{2}u^2}-\textstyle{\binom{n}{3}u^3}+\textstyle{\binom{n}{4}u^4}-bu^4\right)\left(1+nu+(\textstyle{\binom{n+1}{2}}-N)u^2\right)\\=(1-Nu^2+\textstyle{\binom{N}{2}}u^4)(1-cu^4) \quad \mod (u^5).
\end{gather*}
Equating the coefficient of $u^4$ in both sides, we have
\[
\textstyle{\binom{n}{2}}(\textstyle{\binom{n+1}{2}}-N)-\textstyle{n\binom{n}{3}}+\textstyle{\binom{n}{4}}-b=\binom{N}{2}-c\,.
\]
To show the desired inequality $b>0$, it suffices to show
\[
\textstyle{\binom{N}{2}+\binom{n}{2}N-\binom{n}{2}\binom{n+1}{2}+n\binom{n}{3}-\binom{n}{4}}<0
\]
or in other words
\begin{equation*}
\label{N}
N^2+(n^2-n-1)N-6\textstyle{\binom{n+2}{4}}<0.
\end{equation*}
Regarding the expression on the left as a quadratic polynomial in $N$, its discriminant can be computed to be $2n^4-2n^2+1$, and hence the expression
\[
\frac{\sqrt{2n^4-2n^2+1}-n^2+n+1}{2}
\]
is the largest of the two roots of this polynomial, while the second root is negative. The assumption in the hypothesis shows that \eqref{N} holds, and thus $b=\varepsilon_{3,4}>0$. 
\end{proof}

Our results, together with the interval for Koszulness recalled in \ref{interval},  suggest that as $N$ ranges from $n+1$ to $n^2/4+n/2$, the rate is decreasing from $(\reg_S(R)+1)/2$ to $1$. We are able to further describe this behavior when assuming a conjecture of Fr\"oberg and L\"ofwall in  \cite{froberg2025generic}.

In \cite{froberg2025generic}, the authors say that a $\kk$-Lie-algebra $L$ has type $(n, r)$ if it is a quotient of the free Lie $\kk$-superalgebra (with
squares of odd elements) on the odd elements $x_1, \dots, x_n$ modulo the Lie ideal generated by linearly independent elements $f_1, \dots, f_r$, where \[
f_i = \sum_{1\le j\le n} c_{i,j}x^2_j
+\sum_{1\le j<l\le n} c_{i,j,l}([x_j , x_l])
\]
where $c_{i,j} , c_{i,j,l}\in \kk$. Similarly, a quotient $A$ of a non-commutative polynomial ring in variables $x_1, \dots, x_n$ modulo the ideal generated by forms $f_1, \dots, f_r$ as above, is called a non-commutative $\kk$-algebra of Lie type $(n,r)$.  Any such algebra $A$ is the universal enveloping algebra of the corresponding Lie algebra $L$. We denote by $L(z)$ and $A(z)$ the Hilbert series of these algebras. They are related by the formula
\[
A(z)=\Exp(L(z))=\prod_{i\ge 0}\frac{(1+z^{2i-1})^{e_{2i-1}}}{(1-z^{2i})^{e_{2i}}}\,,
\]
where $L(z)=\sum_{i\ge 0}e_iz^i$. A noncommutative $\kk$-algebra of Lie type $(n,r)$ is said to be {\it generic} if it has minimal Hilbert series among the noncommutative $\kk$-algebras of Lie type $(n,r)$ (where the field $\kk$ is allowed to vary, but its prime field is fixed), with the minimum being attained when the coefficients $c_{i,j}$, $c_{i,j,l}$ are algebraically independent over the prime field of $\kk$. Similarly, a $\kk$-Lie algebra of type $(n,r)$ is said to be {\it generic} if it has minimal Hilbert series among the $\kk$-Lie algebras of type $(n,r)$, in the same sense.  In \cite[Proposition 2.5]{froberg2025generic} it is noted that the Lie algebra $L$ is generic if and only if its universal enveloping algebra $A$ is generic.

The following conjecture is proposed in \cite{froberg2025generic}.

\begin{conjecture}\label{Fconj}
    If $L$ is a generic Lie algebra of type $(n,r)$ then  
    \[
    L(z)=\big[\mathrm{Log}(1/(1-nz+rz^2)\big]\,,
    \]
    where for a power series $V(z)$ with constant term $1$, $\Log$ (the inverse of $\Exp$) is defined by
    \[
    \Log(V(z))=\sum_{i\geq 1}\frac{\mu(i)}{i}\log(V((-1)^{i+1}z^i))
    \]
 and $\mu$ denotes the M\"{o}bius function and $\log$ is the natural logarithm.
\end{conjecture}

For the purpose of the next result, we write $S_n=\kk[x_1, \dots, x_n]$; the number of variables is no longer fixed. 

\begin{corollary}
\label{c:limit}
    Assume \cref{Fconj}. Let $\alpha\in \mathbb{N}$. For each $n\ge 1$, set \[
    N_n=\frac{n^2}{4}+\frac{n}{2}-\alpha
    \]
    and let $R_n=S_n/(f_1, \dots, f_{N_n})$ be a quotient of $S_n$ of type $(N_n;2)$ such that the forms $f_1, \dots, f_{N_n}$ have generic coefficients. Then  
    \[
    \lim_{n\to\infty}\rate(R_n)=1 \qand \lim_{n\to\infty}\slant(R_n)=1. 
    \]
 \end{corollary}   
\begin{proof}
    For $n\gg 0$, we have $N_n\geq \frac{{n+2\choose 3}}{n}$, which implies that $(R_n)_{\geqslant 3}=0$. Then \cite[Corollary  1.3]{MR846457}  gives that the Koszul dual $A$ of $R$ is a $\kk$-noncommutative algebra of Lie type $(n,c)$ with 
    \[
    c=\binom{n+1}{2}-N_n=\frac{n^2}{4}+\alpha.
    \]
    and is the universal enveloping algebra of a $\kk$-Lie algebra $L$ of type $(n,c)$. By \cite[Theorem 4.1]{froberg2025generic}, $A$ is generic, and hence $L$ is also generic. 
    Using \cref{Fconj}, we have:  
    $$ A(z) = \Exp(L(z))=\Exp\Big(\big[\mathrm{Log}(1/(1-nz+cz^2)\big]\Big).$$
    
 Let $a_n:=\frac{n}{2}+i\sqrt{\alpha}$. We have $1-nz+cz^2=(1-a_nz)(1-\overline{a_n}z)$, and 
 \begin{align*}
\mathrm{Log} \left( \frac{1}{1-nz+cz^2}\right)&=\sum_{r=1}^{\infty} \frac{\mu(r)}{r}\left( \mathrm{log} \left(\frac{1}{(1-(-1)^{r+1}a_nz^r)}\right) + \mathrm{log} \left(\frac{1}{(1-(-1)^{r+1} \overline{a_n} z^r)}\right)\right)\\
&=\sum_{r=1}^{\infty} \frac{\mu(r)}{r} \left(  \sum_{m=1}^{\infty} (-1)^{(r+1)m}\frac{(a_n^m+\overline{a_n}^m)z^{rm} }{m}  \right)\\
&=\sum_{k=1}^{\infty}(-1)^k\frac{1}{k} \left( \sum_{r\, | \, k} (-1)^{k/r} \mu(r) (a_n^{k/r}+\overline{a_n}^{k/r})\right)z^k\,.
\end{align*}

The coefficient of $z^k$ in this series is
\begin{equation}\label{L_coeff}
L_k:=\frac{1}{k} \sum_{r\, | \, k}(-1)^{k+k/r} \mu (r) (a_n^{k/r} + \overline{a_n}^{k/r})\,.
\end{equation}
Let $\nu(n)$ denote the smallest value of $k$ for which $L_k<0$. 
Then
$$
\left[\mathrm{Log}\left(\frac{1}{1-nz+cz^2}\right)\right]=\mathrm{Log}\left(\frac{1}{1-nz+cz^2}\right) + z^{\nu(n)} B(z)
$$
where $B(z)=\sum_{i=0}^{\infty} b_iz^i$ has $b_0\ne 0$.
It follows that 
$$
A(z)=\frac{1}{1-nz+cz^2} + z^{\nu(n)}C(z)\,,
$$
 where $C(z)=\sum_{i=0}^{\infty} c_iz^i$ has $c_0\ne 0$.   
    Since $H_R(-z)=1-nz+cz^2$, $\nu (n)$ is also the smallest positive exponent of a power of $z$ that has a nonzero coefficient in $A(z)H_R(-z)$.
   Equation \eqref{e:a} established in the proof of \cref{firstdrop} now gives 
$\rate(R)=\frac{\nu(n)-1}{\nu(n)-2}$ and $\slant(R)=\frac{\nu(n)}{\nu(n)-1}$
and the conclusion will follow once we prove that $\lim_{n\to \infty} \nu(n)=\infty$.

For a fixed $k$, the right-hand side of (\ref{L_coeff}) is a polynomial in $n$ of degree $k$ with positive leading coefficient. Indeed, the terms corresponding to $r>1$ have degree less than $k$ and the term corresponding to $r=1$ is $\frac{1}{k}((\frac{n}{2}+i\sqrt{\alpha})^k+(\frac{n}{2}-i\sqrt{\alpha})^k)$; the leading term is $\frac{1}{k2^{k-1}}$. It follows that there exists $N(k)$ such that for $n\ge N(k)$, $L_0, \ldots, L_k$ are positive, hence $\nu(n)>k$.
\end{proof}

\section{Asymptotic behavior of maximal shifts}
\label{s:asymptotic}

So far, our computations of rate and slant in special cases show that these invariants are often attained at small indices, and hence they can be read from the first few terms of the resolution of $\kk$. Such computations do not provide thus a
complete picture of how the Betti table behaves asymptotically. In this section, we shed more light on the asymptotic behavior of maximal shifts in the case of rings that are either Golod or are homomorphic images of a quadratic complete intersection via a Golod homomorphism. 

We first give a version of a classical problem, cf.~\cite{gilmore1966theory}. 

\begin{setting}[{\bf Unbounded Knapsack Problem}]\label{settingUKP}
 Let $\mathcal N$ denote a (possibly infinite) set of integers. 
Let $\{j_k\}_{k\in \mathcal{N}}$ be a set of non-negative integers and $\{i_k\}_{k\in \mathcal{N}}$ to be a set of distinct positive integers.  Set 
\[
\tau=\sup\;\{j_k/i_k:\;k\in \mathcal N\}
\]
and assume that the supremum is attained. 
Let $a\in \mathcal N$ such that $j_a/i_a=\tau$ and if $j_k/i_k=\tau$ for some $k\ne a$ then $i_k>i_a$. For $d>0$ consider 
\[
\Lambda(d)=\max_{x_k\in \mathbb N}\quad  \sum_{k\in \mathcal N} x_kj_k \quad \subjectto\quad  \sum_{k\in \mathcal N} x_ki_k\le d.
\] 
Observe that these sums are finite, as the inequality $\sum_{k\in \mathcal N} i_kx_k\le d$ forces all but finitely many values of $x_k$ to be zero. 
\end{setting}

\begin{lemma}
\label{l:UKP}
    Assume \cref{settingUKP}. Then \[\Lambda(d+i_a)=\Lambda(d)+\Lambda(i_a)\qforall d\gg 0.\]   
\end{lemma}
\begin{proof}
First, observe that 
\begin{equation}
\label{lambda}
\Lambda(d)\le\tau d\qforall d>0\,.
\end{equation}
Indeed, we have $j_k\le i_k\tau$ for all $k\in \mathcal N$, and hence for $x_k\in \mathbb N$ with $\sum x_ki_k\le d$ we have
\[
\sum_{k=1}^\infty x_kj_k\le \tau\sum_{k=1}^\infty x_ki_k\le \tau d\,.
\]
Additionally,
\begin{equation}
\label{lambda2}
  \Lambda(d+i_a)\ge \Lambda(d)+j_a \qforall d>0.
\end{equation}
Indeed, if $\{x_k\}_{k\in \mathcal N}$ attain the maximum for $\Lambda(d)$, then take $x_k'=x_k$ if $k\ne a$ and $x_a'=1$, so that
\[
\sum x_k'i_k=\sum x_ki_k+i_a\le d+i_a\qand \sum x_k'j_k=\sum x_kj_k+j_a=\Lambda(d)+j_a,
\]
yielding the inequality \eqref{lambda2}. 

Fix an integer $0\le i<i_a$. 
    As in \cite[Equation (13)]{gilmore1966theory}, define the function 
    \[
    \delta_i(q)=(qi_a+i)\tau-\Lambda(qi_a+i) \qfor q\in \mathbb Z_{>0}.
    \] 
Note that $\delta_i(q)\ge 0$ for all $q> 0$ by \eqref{lambda}.   Using \eqref{lambda2} we have $$\delta_{i}(q+1)-\delta_i(q)=\Lambda(qi_a+i)+j_a-\Lambda(qi_a+i+i_a)\leq 0$$
 for all $q>0$,  therefore $ \delta_i$ is non-increasing.
    
 By taking $x_a=q$ and $x_k=0$ for $k\ne a$ in the definition of $\Lambda$ we get
\begin{equation}
\label{lambda4}
\Lambda(qi_a+i)\ge qj_a
\end{equation}
and then 
    \[
    \delta_i(q)=(qi_a+i)\tau-\Lambda(qi_a+i)\le (qi_a+i)\frac{j_a}{i_a}-qj_a=\frac{j_ai}{i_a}\leq \frac{j_a(i_a-1)}{i_a}.
    \]
    Since $\delta_i(q)$ is a fraction of the form $b/i_a$ for some $b\in \mathbb{N}$, we see that $\delta_i(q)$ takes only finitely many values as $q$ varies. 
    Since $\delta_i$ is non-increasing, it must be eventually constant, so there exists an integer $N_i$ such that  $\delta_i(q)=\delta_i(q+1)$ for all $q\ge N_i$,  which then implies  
    \[
\Lambda(qi_a+i+i_a)=\Lambda(qi_a+i)+\Lambda(i_a) \qforall q\ge N_i.
\]
Here, we have used the fact that $\Lambda(i_a)=j_a=\tau i_a$, which comes from \eqref{lambda} and \eqref{lambda4} with $i=0$. 

Let $d$ be an integer and write $d=qi_a+i$ for integers $q,i$ with $0\le i<i_a$. 
Setting $N=\max\{i_1, N_0, N_1, \dots, N_{i_a-1}\}$, we get $\Lambda(d+i_a)=\Lambda(d)+\Lambda(i_a)$ for all $d\ge N$. 
\end{proof}

The next proposition provides the main ingredient for the study of the asymptotic behavior of maximal shifts of $R$ in cases when we know that $\Po_\kk^R$ is rational and satisfies certain conditions. 

\begin{proposition}
\label{p:general-additivity}
Let $\Qo_1,\Qo_2, \Po\in \mathcal P$ such that $\Po=\Qo_1\Qo_2^{-1}$ and the following hold: 
\begin{enumerate}
\item $\Qo_1=\sum_{i\ge 0}a_i(tu)^i$ with
\begin{enumerate}
    \item $a_i>0$ for all $i\ge 0$; or \item there exists an integer $r\ge 1$ such that $a_i>0$ for all $i\le r$ and $a_i=0$ for all $i>r$. 
\end{enumerate}
\item $\slant(\Qo_2)$ is attained for the first time at $i^*$ for some $i^*\in \mathbb N$. 
\item When writing $1-\Qo_2=\sum_{i\ge 1}f_i(u)t^i$, the leading coefficient of $f_i(u)$ is non-negative for all $i\ge 1$ and $f_1(u)=0$, $f_2(u)\ne 0$. 
\end{enumerate}
Then $t_i(\Po)=t_{i-i^*}(\Po)+t_{i^*}(\Po)$ for all $i\gg 0$. 
\end{proposition}

\begin{proof}
Set $\mathcal N=\{k\in \mathbb N\colon f_k(u)\ne 0\}$. Observe that $1\notin \mathcal N$ and $2\in \mathcal N$, by hypothesis. 
With $\Po=\sum_{i\ge 0}D_i(u)t^i$, expanding the fraction $\Qo_1\Qo_2^{-1}$ as a series yields
\[
D_i(u)=\sum_{\substack{p+q=i\\p,q\ge 0}}a_pu^p\Big(\sum_{\substack{x_k\in \mathbb N\,\,\forall\,\, k\in \mathcal N\\ \sum_{k\in \mathcal N}kx_k=q}}\big(\prod_{k\in \mathcal N} f_k(u)^{x_k}\big)\Big)=\sum_{\substack{x_k\in \mathbb N\,\,\forall\,\, k\in \mathcal N\cup\{1\}\\ \sum_{k\in \mathcal N\cup\{1\}}kx_k=i}}a_{x_1}u^{x_1}\big(\prod_{k\in \mathcal N} f_k(u)^{x_k}\big).
\]
Set $j_1=1$ and $j_k=\deg(f_k)$ for $k\in \mathcal N$. 

Assume first $a_i> 0$ for all $i\ge 0$. Then, in view of our hypotheses (particularly (3), which ensures that there are no cancellations that will lower the degree in the summation in the expression of $D_i$), we have: 
\[
\deg(D_i)=\max_{x_k\in \mathbb N}\sum_{k\in \mathcal N\cup\{1\}} j_kx_k \quad \subjectto \quad \sum_{k\in \mathcal N\cup\{1\}}kx_k=i.
\]
We claim that this maximum is equal to 
\[
M_i=\max_{x_k\in \mathbb N}\sum_{k\in \mathcal N\cup \{1\}} j_kx_k \quad \subjectto \quad \sum_{k\in \mathcal N\cup\{1\}}kx_k\le i.
\]
Indeed, the inequality $\deg(D_i)\le M_i$ is clear. Assume that $\{x_k\}_{k\in \mathcal N\cup\{1\}}$ are such that the maximum is attained in the definition of $M_i$. To show the reverse inequality, we will show $\sum_{k\in \mathcal N\cup\{1\}}kx_k=i$. If $\sum_{k\in \mathcal N\cup\{1\}}kx_k<i$, then with $x'_k=x_k$ if $k>1$ and $x'_1=x_1+1$ we have 
\[
\sum_{k\in \mathcal N\cup\{1\}}kx_k'=1+\sum_{k\in \mathcal N\cup\{1\}}kx_k\le i\qand \sum_{k\in \mathcal N\cup\{1\}}j_kx_k'=M_i+j_1=M_i+1>M_i
\]
a contradiction. 

Assume now that $a_i> 0$ for all $i\le r$ and $a_i=0$ for $i>r$. Our hypotheses give
\[
\deg(D_i)=\max_{x_k\in \mathbb N}\sum_{k\in \mathcal N\cup \{1\}} j_kx_k \quad \subjectto \quad \sum_{k\in \mathcal N\cup \{1\}}kx_k=i \qand x_1\le r.
\]
We claim that this maximum is also equal to $M_i$. 
The inequality $\deg(D_i)\le M_i$ is clear. Conversely, assume that $x_k$ are such that the maximum is attained in the definition of $M_i$. As seen above, we must have $\sum_{k\in \mathcal N\cup\{1\}}kx_k=i$. If $x_1>r$, then let $q=\Big\lceil \frac{x_1-r}{2}\Big\rceil$, so that $2q\ge x_1-r>2q-1$, implying $0\le r-1< x_1-2q\le r$. Then, if we take $x_1'=x_1-2q$, $x_2'=x_2+q$ and $x_k'=x_k$ for $k>2$, we have
\begin{align*}
\sum_{k\in \mathcal N\cup \{1\}}j_kx_k'&=(x_1-2q)+j_2(x_2+q)+\sum_{k\in \mathcal N\smallsetminus\{2\}}j_kx_k\ge \sum_{k\in \mathcal N\cup \{1\}}j_kx_k=M_i \\
\sum_{k\in \mathcal N\cup \{1\}}kx'_k&=\sum_{k\in \mathcal N\cup \{1\}}k x_k=i\qand x'_1\le r
\end{align*}
where we used $j_2\ge 2$, which is a consequence of $\Po\in \mathcal P$.  
Consequently, we must have $\deg(D_i)\ge M_i$ and we conclude equality must hold. 

In both cases, we have $t_i(\Po)=\deg(D_i)=M_i$. The conclusion follows from \cref{l:UKP}. 
\end{proof}

\begin{corollary}
\label{limit}
Assume the hypotheses of \cref{p:general-additivity}. Then 
$$
\slant(\Po)=\lim_{i\rightarrow\infty}\frac{t_i(\Po)}{i}=\lim_{i\rightarrow\infty}\frac{t_i(\Po)-1}{i-1} .$$ 
\end{corollary}
\begin{proof}
    Set $j^*=t_{i^*}(\Po)$. From \cref{p:general-additivity} we know that there exists $N\in \mathbb{N}$ such that if $i>N$ then $t_{i+i^*}(\Po)=t_i(\Po)+j^*$. Pick an integer $\alpha>N$. We have  $$\lim_{n\rightarrow\infty}\frac{t_{\alpha +ni^*}(\Po)}{\alpha +ni^*}= \lim_{n\rightarrow\infty}\frac{t_{\alpha}(\Po)+nj^*}{\alpha +ni^*}=\frac{j^*}{i^*}$$
    and since the limit is the same for all $\alpha$ we have
\[
\lim_{i\to \infty}\frac{t_i(\Po)}{i}=\frac{j^*}{i^*}=\slant(\Po)\,.
\]
Then 
\[
\lim_{i\to \infty}\frac{t_i(\Po)-1}{i-1}=\lim_{i\to \infty}\left(\frac{t_i(\Po)}{i}-\frac{1}{i}\right)\cdot \frac{i}{i-1}=\lim_{i\to \infty}\frac{t_i(\Po)}{i}\,. \qedhere
\]
\end{proof} 

\begin{theorem}
\label{t:additivity}
Assume there exists a surjective Golod homomorphism  $P\to R$, where $P=S$ (thus $R$ is Golod) or $P=S/J$ is a quadratic complete intersection ring. 
\begin{enumerate}
\item There exists an integer $i^*$ such that \[
t_i^R(\kk)=t_{i-i^*}^R(\kk)+t_{i^*}^R(\kk) \qforall i\gg 0. 
\]
\item ${\displaystyle \slant(R)=\lim_{i\rightarrow\infty}\frac{t_i^R(\kk)}{i}=\lim_{i\rightarrow\infty}\frac{t_i^R(\kk)-1}{i-1}}$.
\end{enumerate}
\end{theorem}

\begin{proof}
Set $\Po=\Po_\kk^R(t,u)$. Since the map $P\to R$ is Golod, we have an equality
\[
\Po=\frac{\Po_\kk^P(t,u)}{1-t(\Po_R^P(t,u)-1)}.
\]
Set $\Qo_2=1-t(\Po_R^P(t,u)-1)$ and $\Qo_1=\Po_\kk^P(t,u)$. Since $\mathcal P$ is a group and $\Po$, $\Qo_1\in \mathcal P$, we have $\Qo_2\in \mathcal P$. Furthermore, $\Qo_1$ satisfies condition (1)(b) of \cref{p:general-additivity} (with $r=n$) when $P=S$ and condition (1)(a) when $P$ is a quadratic complete intersection. 

When $\slant(\Po)=1$, the statement is clear. 
 Assume now $\slant(\Po)>1$. We know from \cref{denom-poly} that  $\slant(\Po)$ is attained. 
 Then, \cref{l:inverse} gives $\slant(\Po^{-1})=\slant(\Po)>1$  and  $\slant(\Po^{-1})$ is attained. Further, since $\slant(\Qo_1)=1$,  \cref{QP} gives $\slant(\Qo_2)=\slant(\Po^{-1}\Qo_1)=\slant(\Po^{-1})>1$ 
and $\slant(\Qo_2)$ is attained. Thus, condition (2) of \cref{p:general-additivity} is satisfied. Also, condition (3) is clearly satisfied. Thus, an application of \cref{p:general-additivity} and \cref{limit} yields the desired conclusions. 
\end{proof}


\section{Maximal Slants for Monomial Ideals}
\label{s:monomial}
In this section, we show that for monomial ideals, the slant is not always minimal, and we prove a sharp upper bound for the slants of monomial ideals.

Berglund \cite{berglund2006poincare} gives an explicit description (in terms of the reduced simplicial homology of a certain simplicial complex) of the Poincar\'{e} series of the residue field when $R$ is defined as a quotient of the polynomial ring by a monomial ideal.
We review this result below.

 Let $\kk$ be a field, $I$ a monomial ideal in $S=\kk[x_1,\ldots,x_n]$ and set $R=S/I$. Let $M({I})$ denote the set of minimal monomial generators of ${I}$. 
 
 \begin{notation}
 Let $G$ be a set of monomials in $S$. The set $G$ has an associated gcd-graph where the elements of $G$ are vertices, and there is an edge between  $p,\;q\in G$ if and only if  $\gcd(p,q)\neq 1$. We introduce notation as follows: 
 \begin{itemize}
 \item $m_G:=$ the LCM of elements of $G$, with the convention that $m_{\emptyset}=1$;
 \item $\mathcal C(G):=$ the set of connected components of the gcd-graph of $G$;
 \item $c(G):= |\mathcal C(G)|$;
 \item $K(I):=$ be the set of all nonempty saturated subsets of $M(I)$, where 
 $G$ is said to be {\it saturated} in $M(I)$ if for all $m\in M(I)$ and all $C\in \mathcal C(G)$, $m|m_C$ implies $m\in G$;
 \item For $G\in K(I)$, define a simplicial complex $\Delta_G$ with vertex set $G$ by
 \begin{gather*}
\Delta_G:=\{F\subseteq G\colon m_F\neq m_G \qor\text{$F\cap C$ is disconnected for some $C\in \mathcal C(G)$} \,;
 \end{gather*}
 \item $\tilde{H}_i(\Delta_G;\kk):=$ the $i$'th reduced homology group for $\Delta_G$ with coefficients in $\kk$;
 \item  $\Tilde{H}(\Delta_G;\kk)(t):=\sum_{i\in \mathbb Z} \dim_\kk\tilde{H}_i(\Delta_G;\kk)t^i$. 
  \end{itemize}
 \end{notation}
\begin{theorem}\cite{berglund2006poincare}\label{berglund-theorem}
With notation as above, we have
$\Po_\kk^R(t, u)=\displaystyle \frac{(1+tu)^n}{\bo _R(t,u)}$, with \begin{equation}\label{bd1}
\bo _R(t,u)=1+\sum_{G\in K(I)}u^{\deg(m_G)}(-t)^{c(G)+2}\Tilde{H}(\Delta_G;\kk)(t).
\end{equation}
 \end{theorem}

 Our first goal is to establish a formula for $\slant(R)$, using the formula \eqref{bd1}. This will be achieved in \cref{mslant}, after establishing below some preliminaries.

For $G_1,\; G_2 \in K(I)$, we define the product of the their simplicial complexes $$\Delta_{G_1}\cdot \Delta_{G_2}=\{F\subseteq G_1\cup G_2: F\cap G_1\in \Delta_{G_1}\text{ or }F\cap G_2\in \Delta_{G_2}\}\,.$$
It is easy to check that $\Delta_{G_1}\cdot \Delta_{G_2}$ is a simplicial complex. The next statement follows directly from the definitions.
\begin{lemma}\label{l6.1}
    For $G\in K(I)$, if $G=G_1\cup G_2$ with $\gcd(m_{G_1},m_{G_2})=1$, then $\Delta_G=\Delta_{G_2}\cdot\Delta
    _{G_2}$.    \qed
\end{lemma}

We set  $\tilde{K}(I)=\{G\in K(I)\colon \Tilde{H}(\Delta_G;\kk)(t)\ne 0\}$ and for each $G\in \tilde{K}(I)$ we set
\[
h(G):=\min\{i\colon \tilde{H}_i(\Delta_G;\kk)\ne 0\} \,.
\]
We define a function $F\colon \tilde{K}(I)\to\mathbb{Q}$ by $$F(G)=\frac{\deg(m_G)}{2+c(G)+h(G)}.$$

\begin{lemma} \label{l6.3}
Set $\mathbf{m}=\max\{F(G)\colon G\in \tilde{K}(I)\}$. If $G\in \tilde{K}(I)$ is such that $F(G)=\mathbf{m}$, then $F(C)=\mathbf{m}$ for all $C\in \mathcal C(G)$. 
\end{lemma}
\begin{proof}Set $r=c({G})$ and let ${G}=C_1\cup C_2\cup \cdots \cup C_r$, where $C_i$ are connected components of the gcd-graph of $G$. Since $G\in K(I)$, we see that $C_i\in K(I)$ for all $i$ as well. With this decomposition, using \cref{l6.1}, we have $\Delta_{G}= \Delta_{C_1}\cdot \Delta_{C_2}\cdots \Delta_{C_r} $
and $m_{G}=m_{C_1} m_{C_2}\cdots m_{C_r}$.
From  \cite[Section 2.1]{berglund2006poincare} we know that \[
\Tilde{H}(\Delta_{G};\kk)(t) = t^{2r-2}\Tilde{H}(\Delta_{C_1};\kk)(t)\cdots \Tilde{H}(\Delta_{C_r};\kk)(t).
\]
Since $\tilde{H}(\Delta_G,\kk)(t)\ne 0$, we have  $\Tilde{H}(\Delta_{C_i};\kk)(t)\neq 0$ for all $i$, and thus $C_i\in \tilde{K}(I)$ for all $i$, and $h(G)=2r-2+\sum_{i=1}^rh(C_i)$. Also, observe that  $\deg(m_{G})=\sum_{i=1}^r\deg(m_{C_i})$ and $c(G)=r=\sum_{i=1}^rc(C_i)$, and hence 
\begin{equation}
\label{FG}
2+c(G)+h(G)=\sum_{i=1}^r\left(2+c(C_i)+h(C_i)\right)\,.
\end{equation}
Since $G$ maximizes $F$, we have
\begin{equation}
\label{FCi}
F(C_i)= \frac{\deg(m_{C_i})}{2+c(C_i) +h(C_i)}\le \frac{\deg(m_{G})}{2+c(G)+h(G)}=F(G) \qforall i\in [r]\,.
\end{equation}
We then have
\begin{align*}
\deg(m_{G})=\sum_{i=1}^r\deg(m_{C_i})\le \deg(m_{G})\cdot \sum_{i=1}^n\frac{(2+c(C_i)+h(C_i))}{2+c(G)+h(G)}=\deg(m_{G})
\end{align*}
where the inequality comes from \eqref{FCi} and the last equality uses \eqref{FG}. The inequality must be an equality, forcing all inequalities in \eqref{FCi} to be equalities. 
\end{proof}

\begin{proposition}\label{mslant}
If $I$ is a monomial ideal of $S=\kk[x_1,\dots, x_n]$ and $R=S/I$, then
\[
\slant(R)=\max\{F(G)\colon G\in \tilde{K}(I)\}=\max\{F(G)\colon  G\in \tilde{K}(I)\,\,{\rm  and }\,\, c(G)=1\}.
\]
\end{proposition}
\begin{proof}
The second equality follows from \cref{l6.3}. We prove the first equality.  
Set $\mathbf{m}=\max\{F(G)\colon G\in \tilde{K}(I)\}$. 
By \cref{denom-poly}, we have $\slant(R)=\slant(\bo_R)$, with $\bo_R$ as in \eqref{bd1}. We need to show thus $\slant(\bo_R)=\mathbf{m}$. Since $\bo_R$ is a polynomial, $\slant(\bo_R)$ is attained. From $\eqref{bd1}$, we see that for all $i$ we have
\begin{align}
\label{tb}
t_i(\bo_R)&=\deg\Big(\sum_{G\in \tilde{K}(I)}(-1)^{c(G)}\dim_\kk(\tilde{H}_{i-c(G)-2}(\Delta_G;\kk))u^{\deg(m_G)}\Big)\,.
\end{align}
Note that $\tilde{H}_{i-c(G)-2}(\Delta_G,\kk)\ne 0$ implies $i\ge 2+c(G)+h(G)$. Thus, if $\slant(\bo_R)$ is attained at $i$ and $t_i(\bo_R)=\deg(m_G)$ for some $G\in \tilde{K}(I)$, we have
$$
\slant(\bo _R)=\frac{t_i(\bo_R)}{i}=\frac{\deg(m_G)}{i}\le\frac{\deg(m_G)}{2+c(G)+h(G)}=F(G)\le \mathbf{m}\,.
$$
Set 
\[
j^*=\min\{\deg(m_G)\colon G\in \tilde{K}(I)\text{ and } F(G)=\mathbf{m}\} \qand i^*=\frac{j^*}{\mathbf m}
\]
so that $\mathbf{m}=j^*/i^*$, and  define
\begin{equation}
\label{U}
\mathcal U=\{ G\in \Tilde{K}(I): \deg(m_G)=j^*\text{ and } F(G)=\mathbf{m} \}\,.
\end{equation}
If the gcd-graph of $G\in \mathcal U$ is not connected and $C$ is one of the connected components of $G$, then \cref{l6.3} implies that $F(C)=\mathbf{m}$. Since $\deg(m_{C})<\deg(m_G)=j^*$, this contradicts the choice of $j^*$. Therefore, for all $G\in \mathcal U$ we have $c(G)=1$.

We claim that $t_{i^*}(\bo_R)=j^*$,  which shows that $\slant(\bo_R)=\mathbf{m}$ (attained at $i^*$). To see that $t_{i^*}(\bo_R)\le j^*$, observe that if $G\in \tilde{K}(I)$ is such that $\tilde{H}_{i^*-c({G})-2}(\Delta_{G},\kk)\ne 0$, then $i^*\ge 2+c(G)+h(G)$, and the inequality $F(G)\le \mathbf{m}$ implies 
     \begin{equation}
     \label{degj}
     \deg(m_G)=F(G)\cdot (2+c(G)+h(G))\le \frac{j^*}{i^*}\cdot  (2+c(G)+h(G))\le j^*. 
     \end{equation}
This shows that the summation in \eqref{tb} does not contain any terms with degree larger than $j^*$. 
To show that $t_{i^*}(\bo_R)=j^*$, we need to show that the coefficient of $u^{j^*}$ in the summation in \cref{tb} is nonzero. Indeed, note that $\deg(m_G)=j^*$ in \eqref{degj} implies   $F(G)=j^*/i^*=\mathbf{m}$, and hence $G\in \mathcal U$. Conversely, if $G\in \mathcal U$, then $\deg(m_G)=j^*$ and  $2+c(G)+h(G)=i^*$. Since $c(G)=1$ for all $G\in \mathcal U$, as established above, we see that the coefficient of $u^{j^*}$ in the summation in \eqref{tb} is 
\[
\sum_{G\in \mathcal U}(-1)^{c(G)}\dim_\kk \tilde{H}_{h(G)}(\Delta_G;\kk)=-\sum_{G\in \mathcal U}\dim_\kk \tilde{H}_{h(G)}(\Delta_G;\kk)\ne 0\,. \qedhere
\]  
\end{proof}
Before moving on to the main result of this section, we need the following lemma.
\begin{lemma}\label{mfb} Let $n,\;d\in \mathbb{N}$, where $n,\;d>1$. Define a function $f:\mathbb{N}\longrightarrow \mathbb{R}$, $$f(x)=\frac{x}{\lceil \frac{x-1}{d-1}\rceil +1}.$$
Then $$\max\{f(x)\colon 1\leq x\leq nd, x\in \mathbb N\}=\frac{(d-1)\lfloor \frac{nd-1}{d-1}\rfloor +1}{n+\lfloor\frac{n-1}{d-1} \rfloor +1}\,.$$
\end{lemma}
\begin{proof}
Let $0\le k \le \lfloor \frac{nd-1}{d-1}\rfloor-1$ and
consider $k(d-1)+1\le x \le (k+1)(d-1)+1$. We claim that the maximum of $f(x)$ for $x$ in this interval is achieved at the right endpoint, $x=(k+1)(d-1)+1$.

To prove the claim, we observe that 
$$
\frac{k(d-1)+1}{k+1}=f(k(d-1)+1)\ge f(k(d-1)+2)=\frac{k(d-1)+2}{k+2}
$$
and for $2\le i \le d-2,$
$$
\frac{k(d-1)+i-1}{k+2}=f(k(d-1)+i)\le f(k(d-1)+i+1)=\frac{k(d-1)+i}{k+2}.
$$
That is, for values of $x$ in the specified interval, $f$ decreases from $k(d-1)+1$ to $k(d-1)+2$, and then increases throughout the rest of the interval.

To find the maximum, we compare the values at the endpoints of the interval:
$$\frac{k(d-1)+1}{k+1}=f(k(d-1)+1)\le 
f((k+1)(d-1)+1)=\frac{(k+1)(d-1)+1}{k+2}.$$

The above calculation also shows that $f$ is an increasing function on the set of right endpoints $\{k(d-1)+1:k\in \mathbb{N}\}$. 
It follows that $\max\{f(x): 1\le x \le nd\}$ is achieved at $x_1:=(d-1)\lfloor \frac{nd-1}{d-1}\rfloor +1$ or at $x_2:=nd$ ($x_1$ is the largest integer of the form $k(d-1)+1$ in the range from $1$ to $nd$). We claim that the maximum is achieved at $x_1$. If $r$ denotes the remainder when $n-1$ is divided by $d-1$, we have $x_1=x_2-r$. Assume $r>0$ (otherwise $x_1=x_2$ and we are done).

We have $\frac{x_1-1}{d-1}=\lfloor \frac{nd-1}{d-1}\rfloor$. Letting $K:=\lfloor \frac{nd-1}{d-1}\rfloor+1=n+\lfloor \frac{n-1}{d-1}\rfloor +1$, we have 
$
f(x_1)=\frac{x_1}{K},
$
and
$f(x_2)=\frac{x_2}{K+1}$, so 
$f(x_2)\le f(x_1)$ is equivalent to $(x_2-x_1)K\le x_1$, i.e. 
\begin{equation}\label{ineqq}
rK \le (d-1)K-d+2\,.
\end{equation}
If $n<d$, we have $K=n+1$, $r=n-1$, and the inequality (\ref{ineqq}) is clearly true.
If $n\ge d$, we have $r\le d-2$ and $K\ge n+2$, which gives
$$
rK \le (d-2)K=(d-1)K-K\le (d-1)K-n-2 \le (d-1)K-d+2
$$
therefore inequality (\ref{ineqq}) holds again.
\end{proof}

\begin{theorem}\label{main1} 
Let $d>1$ be an integer. Let $I$ be a monomial ideal in  $S=\kk[x_1, \ldots, x_n]$ with $I\subseteq S_{\geqslant 2}$ and $m(I)\leq d$,  and set $R=S/I$.
Let $\ell$ and $r_0$ denote the quotient and remainder, respectively, when $n-1$ is divided by $d-1$. 
Then 
\begin{equation}\label{maxslant1}\mathrm{slant}(R)\le \displaystyle \frac{nd-r_0}{n+\ell +1} =\frac{(d-1)\lfloor \frac{nd-1}{d-1} \rfloor +1}{n+\lfloor \frac{n-1}{d-1}\rfloor +1}.\end{equation}

Equality  is achieved in (\ref{maxslant1}) if
$
I=(x_1^d, \ldots, x_{n-1}^d, m_0, \ldots, m_{\ell})
$
where $m_j=x_{jd-j+1}\cdots x_{jd-j+d}$ for $0\le j \le \ell -1$ and $m_{\ell}=x_{n-r_0}\cdots x_{n-1}x_n^{d-r_0}.$
\end{theorem}

\begin{proof} Observe that $ r_0 = (n-1)-\lfloor \frac{n-1}{d-1} \rfloor(d-1)$ and $\ell =\lfloor \frac{n-1}{d-1} \rfloor$. 

Let $b_R(t,u)$ be the polynomial defined in \cref{bd1} and let $i_0$ be such that
$\slant(\bo _R)$ is attained for the first time at $i_0$. Note that $i_0\ne 1$, because $t_1(\bo_R)=-\infty$, as noted in the proof of  \cref{denom-poly}. \cref{denom-poly} implies that $\slant(R)$ is also attained at $i_0$ (for the first time, when $R$ is not Koszul).  Let $j_0=t_{i_0}^R(k)$. Since monomial ideals have minimal rate, we have $\displaystyle \frac{j_0-1}{i_0-1} \le \mathrm{rate}(R)= d-1.$ This implies 
\begin{equation}\label{fraction1} i_0\ge \lceil \frac{j_0-1}{d-1} \rceil + 1, \ \ \mathrm{
and \ therefore}\ 
\ 
\frac{j_0}{i_0}\le \frac{j_0}{\lceil \frac{j_0-1}{d-1}\rceil +1}.
\end{equation}
Since $j_0=\mathrm{deg}(m_G)$ for some set $G$ of monomial generators of $I$, and $m(I)\le d$, we have $j_0\le nd$. \cref{mfb} now gives the desired upper bound for $\mathrm{slant}(R)$. 

Now, we will show that these upper bounds for the slants are actually achieved. Set $G_{n,d}:=\{x_1^d, \ldots, x_{n-1}^d, m_0, \ldots, m_{\ell}\}$.  If we show \begin{equation}\label{maxsaturated1}
 F(G_{n,d})=\displaystyle \frac{(d-1)\lfloor \frac{nd-1}{d-1}\rfloor +1}{n+\lfloor \frac{n-1}{d-1} \rfloor +1}    
\end{equation}
then \eqref{maxslant1} and \cref{mslant} will imply $I$ achieves the maximum slant.

We have $m_{G_{n,d}}=x_1^d\cdots x_{n-1}^d\cdot x_n^{d-r_0}$ and $\deg(m_{G_{n,d}})=nd-r_0=(d-1)\lfloor \frac{nd-1}{d-1}\rfloor +1$. It is clear that the gcd-graph of $G_{n, d}$ is connected, and thus $c(G_{n,d})=1$.

We claim that every proper subset of $G_{n, d}$ is a face of $\Delta_{G_{n, d}}$. 
Indeed, if $G'\subseteq G_{n, d}$ does not contain one of $x_1^d, \ldots, x_{n-1}^d,$ or $m_{\ell}$, then $m_{G'}\ne m_{G_{n, d}}$, and if $G'$ does not contain one of $m_1, \ldots, m_{\ell -1}$ then the gcd-graph of $G'$ is disconnected. Also, $G_{n,d}\notin \Delta_{G_{n,d}}$, and hence \begin{equation}
    \tilde{H}_i(\Delta_{G_{n,d}};k)=\begin{cases}
     0 & \text{ if } i\neq |G_{n,d}|-2\\
     1 & \text{ if }i=|G_{n,d}|-2
    \end{cases}
\end{equation} Thus, $h(G_{n, d})=|G_{n, d}|-2=n+\ell -2=n+\lfloor \frac{n-1}{d-1} \rfloor-2$, hence  \eqref{maxsaturated1} holds.
\end{proof}

A natural question is whether the ideal $I$ from \cref{main1}, for which the slant of $R=S/I$ is maximal, is unique with this property. We are able to prove such a statement for the case $n=d$ is an odd integer.

\begin{theorem}\label{T6.8}
     Let $I$ be a monomial ideal in $S=\kk[x_1, \dots, x_n]$ with $I\subseteq S_{\geqslant 2}$,  $m(I)\leq n$ and $n$ is odd. If $R=S/I$ has maximum possible slant (given by \cref{maxslant1}), then $$I=(x_1^n, \ldots, x_{n}^n, x_1\cdots x_n).$$ 
\end{theorem}
\begin{proof}
By \cref{mslant}, there exists $G\in \tilde{K}(I)$ such that $\slant(R)=F(G)$ and $c(G)=1$. Using the expression for $\slant(R)$ in \cref{main1}, we have thus
\[
\frac{\deg(m_G)}{2+c(G)+h(G)}=\slant(R)=\frac{n^2}{n+2}\,.
\]
Since $n$ is odd, we have $\gcd(n^2,n+2)=1$. This implies $\deg(m_G)=n^2$ and $2+c(G)+h(G)=n+2$. Now, $\deg(m_G)=n^2$ implies that $\{x_1^n,\ldots ,x_n^n \}\subseteq G$  and $c(G)=1$ implies that $h(G)=n-1$. 

Suppose $ x_1\cdots x_n\notin G$. 
We claim that for all $E\subseteq G$ with $|E|=n+1$, we must have $E\in \Delta_G$. Indeed, if $E$ does not contain all of $x_1^n, \ldots, x_n^n$, then $m_E\ne m_G$. If $E=\{x_1^n, \ldots, x_n^n, m\}$ where $m\neq x_1\cdots x_n$ then $E$ is disconnected. We conclude that $\Tilde{H}_{n-1}(\Delta_G;\kk)=0$, which contradicts the fact that $h(G)=n-1$. Thus, we must have $x_1\cdots x_n \in G$. 

For the sake of contradiction, suppose $G\neq \{x_1^n,\ldots,x_n^n,x_1\cdots x_n\}$. Then we have $ G=\{x_1^n,\ldots,x_n^n,Y_0,Y_1,\ldots,Y_l\}$, where $Y_i\in M(I)$ and $l\geq 1$ and $Y_0=x_1\cdots x_n$. First, observe that $\Delta_G$ has the following description:
\[
  \Delta_G=\{E\subseteq G: |E|\leq n \text{ or } |E| = n+1\text{ with } Y_0\notin E\}\cup \{E\in \Delta_G:\; |E|\geq n+2\}.
\]
Consider the Alexander dual of $\Delta_G$, which is defined as $\Delta^*_G:=\{ E \subseteq G : G\smallsetminus E\notin \Delta_G\}$. We claim that $\Delta^*_G$ is a simplical complex on vertex set $\{Y_0,Y_1,Y_2,\cdots,Y_l\}$. Indeed, for $E\subseteq G$, if $x_i^{n}\in E$ for some $i$ then $m_{G\smallsetminus E}\neq m_{G}$, which implies that $G\smallsetminus E\in \Delta_G$ and therefore $E\notin \Delta_G^*$. 

Observe that $E\in \Delta^*_G$ implies  $G\smallsetminus E\notin \Delta_G$ and hence $|E|\leq l$. If $E\subseteq \{Y_0,\cdots,Y_{l}\}$ with $|E|=l$, then $G\smallsetminus E$ is disconnected iff $Y_0\in E$. Therefore the only face of size $l$ in $\Delta^*_G$ is $E=\{Y_1,\cdots,Y_l\}$, and in particular  $\tilde{H}_{l-1}(\Delta^*_G;\kk)=0$. Since, $\kk$ is a field, by the universal coefficient theorem, we have $\tilde{H}_{l-1}(\Delta^*_G;\kk)\cong \tilde{H}^{l-1}(\Delta^*_G;\kk)$ and from \cite[Theorem 1.1]{bjorner2009note} we have $\tilde{H}^{l-1}(\Delta^*_G;\kk)\cong \tilde{H}_{n-1}(\Delta_G;\kk)$. Therefore, $\tilde{H}_{n-1}(\Delta_G;\kk)=0$, and this contradicts the equality $h(G)\neq n-1$. Therefore $G=\{x_1^n,\ldots,x_n^n,x_1\cdots x_n\}$. Now for $\{x_1^n,\ldots,x_n^n,x_1\cdots x_n\}$ to be saturated in $M(I)$, we must have $I=( x_1^n, \ldots, x_{n}^n, x_1\cdots x_n )$.
\end{proof}

    Experimental evidence suggests that the statement in Theorem \ref{T6.8} is also true when $d=n$  is even.
    However, the next examples show that for $n\ne d$ there are several different monomial ideals that achieve the maximal slant. 
      Notation from \cref{main1} will be used.

\begin{example}
    Let $d\neq n$, and $d-1|n-1$. Set $ \mathcal{I}=(x_1^d,\ldots,x_n^d,m_0,\ldots,m_{\ell-1} )$ and $\mathcal{J}=( x_1^d,\ldots,x_n^d,m_0,\ldots,m_{\ell-1},Y )$, where $Y=x_1\cdots x_{d-1}x_{d+1}$. From Theorem \ref{main1}, we know $\mathcal{I}$ has maximal slant. To show $\mathcal{J}$ has maximal slant, it is enough to show that for the saturated set $G=M(\mathcal{J})$ we have $h(G)=n+\ell-2$. Observe that $\Delta_G$ has the following description:
    \[
    \Delta_{G}=\{E\subseteq G:\;\; |E|\leq n+\ell\text{ and } E\notin \{ G\smallsetminus m_0,G\smallsetminus Y \} \}\,.
    \]
    Clearly, $\Delta^*_G=\{\emptyset,\{ m_0\},\{Y\}\}$, implying $\tilde{H}^{0}(\Delta^*_G;\kk)\neq 0$, and from \cite{bjorner2009note} we have $\tilde{H}^{0}(\Delta^*_G;\kk)\cong \tilde{H}_{n+\ell-2}(\Delta_G;\kk)$, which gives $\tilde{H}_{n+\ell-2}(\Delta_G;\kk)\neq 0$ and $h(G)\leq n+\ell-2$. Since all $n+\ell-2$ faces are in $\Delta_G$ we have $h(G)\geq n+\ell-2$, therefore $h(G)=n+\ell-2$. Thus $\mathcal{J}$ has maximal slant.
\end{example}
\begin{example}
    Let $d\neq n$ and $d-1 \nmid n-1$, then in $k[x_1,x_2,\ldots,x_n]$, the ideals $\mathcal{I}=(x_1^d,\ldots,x_n^d,m_0,\ldots,m_{\ell})$ and $\mathcal{J}= ( x_1^d,\ldots,x_{n-1}^d,m_0,\ldots,m_{\ell})$, both have maximal slant, as the set $\{ x_1^d,\ldots,x_{n-1}^d,m_0,\ldots,m_{\ell}\}$ is a saturated set in both $M(\mathcal{I})$ and $M(\mathcal{J})$.
\end{example}


\bibliographystyle{abbrv}
\bibliography{references}

\end{document}